\documentclass{article}

\usepackage{graphicx}
\usepackage{amsmath,amsfonts,amssymb,mathtools,amsthm}
\usepackage{url}
\usepackage{comment}
\usepackage{color}
\usepackage{subcaption}
\usepackage{algorithm}
\usepackage[noEnd=true]{algpseudocodex}
\usepackage{float}
\usepackage[shortlabels]{enumitem}
\usepackage{multirow}
\usepackage{diagbox}
\usepackage{hyperref}
\usepackage[noabbrev]{cleveref}
\crefformat{algorithm}{Algorithm~#2#1#3}
\crefrangeformat{algorithm}{Algorithms~#3#1#4 to~#5#2#6}
\crefmultiformat{algorithm}{Algorithms~#2#1#3}%
       { and~#2#1#3}{, #2#1#3}{ and~#2#1#3}

\usepackage{todonotes}
 \usepackage{tabularray}
\usepackage[font=small]{caption}
\newtheorem{theorem}{Theorem}[section]
\newtheorem{proposition}[theorem]{Proposition}
\newtheorem{lemma}[theorem]{Lemma}

\theoremstyle{definition}
\newtheorem{Definition}[theorem]{Definition}
\newtheorem{example}[theorem]{Example}

\def\tto{\;{\lower 1pt \hbox{$\rightarrow$}}\kern -10pt
	\hbox{\raise 2pt \hbox{$\rightarrow$}}\;}

\def\Bar{\overline}
\def\ra{\rangle}
\def\la{\langle}

\def\R{\mathbb{R}}

\def\ox{\bar{x}}

\def\dom{\mbox{\rm dom}}

\def\O{\Omega}
\def\ph{\varphi}

\begin{document}

\title{Qualitative Analysis and Adaptive Boosted DCA for Generalized Multi-Source Weber Problems}

\author{V. S. T. Long, N. M.  Nam, T. Tran, N. T. T.  Van}

\maketitle

\begin{abstract}
	This paper has two primary objectives. First, we investigate  fundamental qualitative properties of the generalized multi-source Weber problem formulated using the Minkowski gauge function. This includes proving the existence of global optimal solutions, demonstrating the compactness of the solution set, and establishing optimality conditions for these solutions. Second, we apply Nesterov's smoothing  and the adaptive Boosted Difference of Convex functions Algorithm (BDCA) to solve both the unconstrained and constrained versions of the generalized multi-source Weber problems. These algorithms build upon the work presented in \cite{Artacho2020, Nam2018}. We conduct a comprehensive evaluation of the adaptive BDCA, comparing its performance to the method proposed in \cite{Nam2018}, and provide insights into its efficiency.~\\[1ex]
\end{abstract}
\section{Introduction}\label{Intro}
The single-source Weber problem (SWP), also known as the Fermat-Weber problem, is an important topic in convex optimization. It involves determining the optimal location of a source (or facility) that minimizes the sum of Euclidean distances to multiple demand points in a finite-dimensional Euclidean space. The SWP has been extensively studied in the field of facility location and is a generalization of the classical Fermat-Torricelli problem, which  involves only three demand points; see, e.g., \cite{Martini2002, Kuhn1973, Mordukhovich2010, Nam2017} and the references therein. However, in practical applications, it is often necessary to find multiple sources to serve a finite number of demand points, leading to what are known as the multi-source Weber problem. This problem involves determining the optimal locations of a finite number of facilities to serve the demand points by minimizing the sum of the total  distances from the sources  to the associated demand points. In the existing literature, the  multi-source Weber problems are often modeled as  nonconvex  nonsmooth optimization problems, which present challenges for traditional optimization techniques.

It has been shown the multi-source Weber problem belongs to the class of DC programming;  see \cite{cuong2023global,Nam2018}. Consequently, optimization techniques and algorithms from DC programming can be employed to investigate the problem from both theoretical and numerical perspectives. One of the most successful algorithms in DC programming, known as the Difference of Convex functions Algorithm (DCA), was initially introduced in \cite{Tao1986} and further developed in \cite{An2007, TA1, TA2, LeThi2018}. The DCA has proven to be robust and efficient across various applications, including the multi-source Weber problem (see, e.g., \cite{TA1, An2005, Nam2018}). By integrating a line search technique into the DCA, a new algorithm called the Boosted Difference of Convex functions Algorithm (BDCA) was introduced in \cite{AragonArtacho2018, AragonArtacho2022}. This line search technique not only accelerates the convergence of the algorithm but also enables it to escape stationary points and reach a global optimal solution in many circumstances. The BDCA has been experimentally shown to outperform the DCA in many practical applications, such as the minimum sum of squares clustering and the multidimensional scaling problem; see  \cite{Artacho2020}.

Inspired by recent advances in DC programming and its applications, this paper explores a general model of the multi-source Weber problem using the Minkowski gauge function. Our objective is to conduct a thorough investigation of the problem's fundamental qualitative properties and to develop effective numerical optimization methods. We begin by proving the existence of global optimal solutions, demonstrating the compactness of the solution set, and establishing optimality conditions for the problem. Next, we apply Nesterov's smoothing technique to reduce the nonsmoothness, thereby creating a new approximate objective function that is favorable for the BDCA. Numerical experiments confirm that the BDCA is highly effective when applied to the generalized multi-source Weber problem. This approach is new, as the BDCA has not yet been applied to this general model of the multi-source Weber problem. Furthermore, we employ the penalty method to address the constrained version of the problem.

The structure of the paper is organized as follows.  \Cref{sec: prelim} introduces  basic notations and definitions used throughout the paper. \Cref{Existence} establishes a global solution existence theorem for the generalized multi-source Weber problem and discusses the properties of the global solution set in detail. This section also characterizes local optimal solutions. \Cref{sec: Multifacility Location} utilizes the algorithm from \cite{Artacho2020} to solve the generalized multi-source Weber problem for both the unconstrained and constrained cases. \Cref{Results} focuses on enhancing the algorithms presented in the previous section and provides several examples comparing the performance of our new algorithms with that of the difference of convex algorithm. Finally, \Cref{Con} summarizes the paper's findings and outlines potential directions for future research.

\section{Preliminaries}\label{sec: prelim}

Throughout this paper, the Euclidean space $\mathbb{R}^n$, where $n\in \mathbb{N}$, is equipped with the inner product $\langle x, y \rangle = \sum_{i=1}^n x_i y_i$ and the norm $\|x\| = (\sum_{i=1}^n x_i^2)^{1/2}$, where $x = (x_1, \ldots, x_n)$ and $y = (y_1, \ldots, y_n)$. However, in Section~\ref{Results}, we use $\|\cdot\|_2$ to denote the Euclidean norm, ensuring consistency with the Taxicab norm and the maximum norm, which are denoted by $\|\cdot\|_1$ and $\|\cdot\|_\infty$, respectively. The closed Euclidean ball  with center $a \in \mathbb{R}^n$ and radius $r \geq  0$ is denoted by $\mathbb B[a, r]$.

Let us first recall basic notions used in the paper. The readers are referred to \cite{An2007,HUL,mordukhovich2023easy,r} and the references therein for more details.
\begin{Definition}
	Let $f\colon \R^n\to (-\infty, \infty]$ be a convex function. An element $v\in\R^n$ is called a \emph{subgradient} of $f$ at $\bar x\in \dom(f)= \{x\in \R^n\; |\; f(x)<\infty\}$ if
	\begin{equation*}
		\langle v,x-\bar x\rangle\leq f(x)-f(\bar x)\; \mbox{\rm for all } x\in \R^n.
	\end{equation*}
	The set of all such elements $v$ is called the \emph{subdifferential} of $f$ at $\bar x$ and is denoted by $\partial f(\bar x)$. If $\ox\not\in \dom(f)$, we set $\partial f(\ox)=\emptyset$.
\end{Definition}

\begin{Definition}
	\label{def_v1} Let $\Omega$ be a nonempty subset of $\R^n$. The {\em distance function} associated with $\Omega$ is defined by
	\begin{equation*}
		d(x; \Omega)=\inf\left\{\|x-w\|\; \big|\; w\in \Omega \right\}, \; x\in \R^n.
	\end{equation*}
\end{Definition}

\begin{Definition}
	Let $F$ be a nonempty compact convex subset of $\mathbb R^n$ with ~$0\in \mbox{\rm int}(F)$. The {\em Minkowski function} associated with $F$  is defined by
	\begin{equation}
		\label{gauge}
		\rho_F(x) = \inf\{t\geq 0\; |\; x\in tF\}, \; x\in \R^n.
	\end{equation}
The {\em polar set} $F^\circ$ of $F$ is the set
	$$F^{\circ} = \{v\in \R^n\mid \langle v, x\rangle \leq 1,\; \forall x\in F\}.$$
	We set
	$\|F\| = \sup\{\|x\|\mid x\in F\}\text{ and }\|F^\circ \| = \sup\{\|v\|\mid  v\in F^\circ\}.$
\end{Definition}

Given a finite set of demand points $A=\{a^1,\ldots,a^m\}\subset \mathbb{R}^n$, the  {\em generalized multi-source Weber problem} (GMWP) can be formulated as
\begin{eqnarray}
	\label{maxminn} \min \left\{f_F(x) = \sum_{i=1}^{m} \left(\min\limits_{\ell=1,\ldots,k} \rho_F\left(x^\ell - a^i\right)\right) \Big{|}\; x = (x^1, \ldots, x^k) \in \left(\R^{n}\right)^k \right\}.
\end{eqnarray}
 We always assume that
\begin{equation}\label{key222}
	m\geq 2\textrm{ and }1 \leq k \leq m.
\end{equation}
We also assume that $F$ is a nonempty compact convex set such that $0\in \mbox{\rm int}(F)$. 

Note that when $F$ is the closed unit ball in $\mathbb{R}^n$, problem~\eqref{maxminn} simplifies to the  \textit{multi-source Weber problem} (MWP), which has been recently studied in \cite{cuong2023global,CTWYOptim, cuong2020qualitative}. 

It can be shown that the GMWP belongs to the class of DC programming. Consider the DC program:
\begin{equation}
	\label{DCP} \min \left\{ f(x) = g(x) - h(x) \mid x \in \mathbb{R}^{n} \right\} \tag{\ensuremath{\mathcal{P}}},
\end{equation}
where $g \colon \mathbb{R}^{n} \to (-\infty, \infty]$ and $h \colon \mathbb{R}^{n} \to \mathbb{R}$ are convex functions.
The function $f$ defined in \eqref{DCP} is called a {\em DC function} and $g-h$ is a {\em DC decomposition} of $f$. Recall that the {\em Fenchel conjugate} of $g$ is defined by
$$g^*(y) = \sup\{\langle y, x\rangle - g(x) \mid x \in \mathbb{R}^{n}\}, \; y\in \R^n.$$
Next, we present the {\em Difference of Convex functions Algorithm} (DCA) (see \cite{Tao1986, TA1,TA2}):
\begin{algorithm}[H]
	\caption{DCA for solving \eqref{DCP}} \label{algDCA}
	\begin{algorithmic}
		\Procedure{DCA\;} {$x_0 \in \mbox{\rm dom}\,g, N \in \mathbb{N}$}
		\For{$p = 0, \dots , N$}
		\State Find $y_p \in \partial h(x_{p})$
		\State Find $x_{p+1}\in\partial g^*(y_p)$ 		\EndFor
		\EndProcedure
		\State \Output{$x_{N+1}$}
	\end{algorithmic}
\end{algorithm}
The {\em boosted DCA} (BDCA) (see  \cite{AragonArtacho2018,Artacho2020}) is an improvement of the DCA by adding a line search step incorporating an Armijo-type condition.  The BDCA can be summarized below:
\begin{algorithm}[H]
	\caption{BDCA for solving \eqref{DCP}} \label{AlgorithmGeneralBDCA}
	\begin{algorithmic}
		\Procedure{BDCA\;}{$x_0\in \dom g$, $\alpha > 0$, $\beta \in (0, 1)$, $N \in \mathbb{N}$}
		\For{$p = 0, \dots , N$}
		\State \underline{Step 1:} Find $y_p \in \partial h(x_{p})$.
		\State Find $z_p\in \partial g^*(y_p)$ or find $z_p$ approximately by solving the problem
		$$\underset{x \, \in \, \mathbb{R}^{nk}}{\operatorname{min}} \;  \ph_p (x) = g(x) - \langle y_p , x \rangle.$$
		\State \underline{Step 2:} Set $d_p = z_p - x_{p}$.  If $d_p = 0$ then Stop. Otherwise, go to Step~3.
		\State \underline{Step 3:} Choose any $\Bar{\lambda}_p \geq 0$, set $\lambda_p =\Bar{\lambda}_p$
		\While{$f(z_p + \lambda_p d_p) > f(z_p) - \alpha \lambda_p^2 \Vert d_p \Vert^2$}
		\State $\lambda_p = \beta \lambda_p$
		\EndWhile
		\State \underline{Step 4:} Set  $x_{p+1} = z_p + \lambda_p d_p$. If $x_{p+1} = x_p$ then Stop.
		\EndFor
		\EndProcedure
		\State \Output{$x_{N+1}$}
	\end{algorithmic}
\end{algorithm}
An advantage of BDCA over DCA is its ability to escape from stationary points and reach to the optimal solution in many circumstances. Numerical examples also shows the improvements in  computational speed. BDCA has proven its effectiveness in the case where both $g$ and $h$ are differentiable in \cite{AragonArtacho2018}, and when $g$ is differentiable but $h$ is not differentiable in \cite{Artacho2020}. The problems in the our paper fall under the latter case. The readers are referred to \cite{Tao1986, AragonArtacho2018, TA1,TA2, Artacho2020} for more details.

In our paper we will use the {\em adaptive BDCA} (aBDCA), in which the self-adaptive step size strategy from \cite{Artacho2020} replaces Step 3 of the BDCA in \cref{AlgorithmGeneralBDCA}. This self-adaptive strategy is detailed in \cref{algorithm: self-adaptive step size}, and allows the trial step sizes to increase or decrease depending on previous iterations.
\begin{algorithm}[H]
	\caption{Self-adaptive trial step size} \label{algorithm: self-adaptive step size}
	\begin{algorithmic}
		\Procedure{Self-Adaptive $\lambda$\;}{$\gamma > 1$, iteration $p$}
		\If{$p<2$}
		\State \Output{Choose some $\overline{\lambda}_{p} >0$}
		\EndIf
		\State Obtain $\lambda_{p-2},\; \lambda_{p-1}$ from Step 3 of BDCA, \cref{AlgorithmGeneralBDCA}, and $ \overline{\lambda}_{p-2}  ,\; \overline{\lambda}_{p-1}$ from previous iterations
		\If{$\lambda_{p-2} = \overline{\lambda}_{p-2}$ \text{ and } $\lambda_{p-1} = \overline{\lambda}_{p-1}$}
		\State $\overline{\lambda}_{p} = \gamma \lambda_{p-1}$
		\Else
		\State $\overline{\lambda}_{p} = \lambda_{p-1}$
		\EndIf
		\EndProcedure
		\Output{$\overline{\lambda}_{p}$}
	\end{algorithmic}
\end{algorithm}

\section{Properties of Optimal Solutions  to the  Generalized Multi-source Weber Problem} \label{Existence}

In this section, we conduct a comprehensive study of the properties of global and local optimal solutions to the GMWP \eqref{maxminn}, extending the related existing results  in \cite{cuong2023global,Nam2018}  and the references therein. 

\subsection{Global Optimal Solutions}

We first establish the existence of global optimal solutions to the GMWP. To proceed, we present the following known results involving the Minkowski function. 

\begin{lemma}
	\label{lemma1} {\rm (See~\cite[Lemma 4.1]{longoptimletter})}
	For any $x\in \R^n$, we have
	\begin{equation}
		\label{estimate_1} \rho_F (x)\leq \|F\|\|F^{\circ}\|\rho_F(-x).
	\end{equation}
\end{lemma}

\begin{proof}
For any $x\in \R^n$, it follows from [\cite[Proposition~2.1(c)]{colombo2004subgradient} that $$\dfrac{\rho_F (x)}{\|F^{\circ}\|}\leq \|x\|\leq\|F\| \, \rho_F(x).$$ Thus, we have
	$$\rho_F (x)\leq \|F^{\circ}\|\|x\|=\|F^{\circ}\|\|-x\|\leq \|F\|\|F^{\circ}\|\rho_F(-x),$$ which yields \eqref{estimate_1}.
\end{proof}

Given $a\in \R^n$ and $r\geq 0$, define the generalized ball
\begin{equation*}
    \mathbb B_F[a; r]=\{x\in \R^n\; |\; \rho_F(x-c)\leq r\}.
\end{equation*}
\begin{lemma}\label{lm2} For any $a\in \R^n$ and $r\geq 0$, the generalized ball $\mathbb B_F[a, r]$ is always compact.   
\end{lemma}
\begin{proof} For any $x\in \mathbb B_F[a, r]$, by the proof of Lemma~\ref{lemma1} we have
\begin{equation*}
  \frac{\|x-a\|}{\|F\|}\leq   \rho_F(x-a)\leq r,
\end{equation*}
which implies that $x\in \mathbb B[a, r\|F\|]$. Thus, $\mathbb B_F[a; r]\subset \mathbb B[a, r\|F\|]$, and hence $\mathbb B_F[a, r]$ is bounded. The closedness of $\mathbb B_F[a, r]$ follows from the continuity of the Minkowski function $\rho_F$; see~\cite[Proposition~6.18]{mordukhovich2023easy}. Therefore, $\mathbb B_F[a; r]$ is compact.     
\end{proof}

We are now ready to prove  the existence of global optimal solutions to~\eqref{maxminn}. 

\begin{theorem}
\label{thm_global} The global optimal solution set $S_F$ of~\eqref{maxminn} is nonempty and closed.
\end{theorem}
\begin{proof}
	Take any $i_0 \in I$, and define
	\begin{equation}
		\label{defrho}
		\rho=\max\limits_{i \in I}\rho_F(a^{i_0}-a^{i}).
	\end{equation}
By Lemma~\ref{lm2} the  generalized ball
	$\mathbb B_F\big[a^{i_0}, (1+\|F\|\|F^\circ\|)\, \rho\big]$ 	is compact.  Using the classical Weierstrass theorem (see, e.g., \cite[Theorem 7.9(i)]{mordukhovich2023easy}), we see that the problem
	\begin{equation}
		\label{compactification}
		\min \left\{f_F\left(x^1, \ldots, x^k\right) \mid x^\ell \in \mathbb B_F\big[a^{i_0}, (1+\|F\|\|F^\circ\|)\, \rho\big]\; \mbox{\rm for all }\;  \ell=1, \ldots, k\right\}
	\end{equation}
	has an optimal solution $\bar{x}=(\bar{x}^1, \ldots, \bar{x}^k) \in \mathbb{R}^{n k}$ satisfying the following inequality $$\rho_F(\bar{x}^\ell-a^{i_0}) \leq (1+\|F\|\|F^\circ\|) \, \rho\text{ for all }\ell=1, \ldots, k.$$
	Next, we will verify that  $\bar{x}$ belongs to $S_F$. Indeed, given any  $x=(x^1, \ldots, x^k)  \in \mathbb{R}^{n k}$, we consider the following two cases:
	\begin{itemize}
		\item[(a)] $\rho_F(x^\ell-a^{i_0}) \leq (1+\|F\|\|F^\circ\|) \rho$ for all $\ell =1, \ldots, k$;
		\item[(b)] $\rho_F(x^{ \ell}-a^{i_0}) > (1+\|F\|\|F^\circ\|)\, \rho$ for some $\ell \in \{1,\ldots,k\}$.
	\end{itemize}
	In case~(a), it is easy to see that $f_F(\bar{x}) \leq f_F(x)$.  In case (b), we first denote by $L$ the set of all indices
	$\ell=1, \ldots, k$ such that
	\begin{equation}
		\label{key66}
		(1+\|F\|\|F^\circ\|) \rho <\rho_F(x^{ \ell}-a^{i_0})
	\end{equation}
	We define a vector $\tilde{x}=(\tilde{x}^1, \ldots, \tilde{x}^k) \in \mathbb{R}^{n k}$ by setting $\tilde{x}^\ell=a^{i_0}$ for $\ell \in L$ and  $\tilde{x}^\ell=x^\ell$ for $\ell \in \{1,\ldots,k\} \backslash L$. Fix any $i \in 1, \ldots, m$. Then, for every $\ell \in L$ we have
	\begin{equation}
		\label{key88}
		\begin{array}{lll}
			\rho_F(\tilde{x}^\ell-a^i) & = \rho_F(a^{i_0}-a^i)                                                                               \\
			                           & = (1+\|F\|\|F^\circ\|) \rho -  \|F\|\|F^\circ\| \rho \;  (\textrm{by}~\eqref{defrho} )            \\
			                           & < \rho_F(x^\ell-a^{i_0}) - \|F\|\|F^\circ\| \rho \;  (\textrm{by}~\eqref{key66} )                 \\
			                           & \leq \rho_F(x^\ell-a^{i_0}) - \|F\|\|F^\circ\|\rho_F(a^{i_0}-a^i)\;  (\textrm{by}~\eqref{defrho}) \\
			                           & \leq \rho_F(x^\ell-a^{i_0}) - \rho_F(a^i-a^{i_0})\;  (\textrm{by Lemma}~\eqref{lemma1})           \\
			                           & \leq \rho_F(x^\ell-a^i),
		\end{array}
	\end{equation}
	where the last inequality holds since $\rho_F$ is subadditive; see~\cite[Theorem 6.14]{mordukhovich2023easy}. Meanwhile, for every $\ell\in \{1,\ldots,k\}\backslash L$, we have
	$$\rho_F(\tilde{x}^\ell-a^i)=\rho_F(x^\ell-a^i). $$ This together with~\eqref{key88} implies that
	\begin{equation}
		\label{leq1}
		f_F(\tilde{x})=\sum_{i=1}^m\min_{\ell=1,\ldots,k}\rho_F(\tilde{x}^\ell-a^i)\leq \sum_{i=1}^{m}\, \min_{\ell=1,\ldots,k}\rho_F(x^\ell-a^i)=f_F(x).
	\end{equation}
	In addition, by the above construction of $\tilde{x}^\ell$, we have $$\tilde x^\ell\in \mathbb B_F[a^{i_0}, (1+\|F\|\|F^\circ\|) \rho]\text{ for all } \ell =1, \ldots, k$$
	and thus $f_F(\bar x)\leq f_F(\tilde{x})$. Combining this with~\eqref{leq1}, we deduce  that $\bar x$ is the global optimal solution to problem ~\eqref{maxminn}. Since $f_F$ is continuous, we conclude that $S_F$ is closed and thus complete the proof of the theorem.
\end{proof}
 
In the next theorem, we show that the global optimal solution set of the gen-
eralized multi-source Weber problem is compact. \begin{theorem}
The set $S_F$ of all global optimal solutions to problem~\eqref{maxminn} is compact.
\end{theorem}
\begin{proof}
By Theorem~\ref{thm_global}, it suffices to show that $S_F$ is bounded. Suppose that it is unbounded. Then, there exists \(y=(y^1, ..., y^k) \in S_F\) such that
\begin{equation}\label{aboveinequality}
\sum_{\ell = 1}^k \|y^\ell\|^2 \geq k\|F\|^2(a + M + 1)^2,
\end{equation}
where $a = \max\{\rho_F(a^1), ..., \rho_F(a^m)\}$, and $M=f_F(y^1, ..., y^k)$ is the optimal value of problem~\ref{maxminn}. 
By \eqref{aboveinequality}, we may assume without loss of generality that
\begin{equation}\label{key35}
	\|y^{k}\| \geq \|F\|(a + M + 1).
\end{equation}
 For any \(i =1, \ldots, m\), by~\cite[Proposition~2.1(c)]{colombo2004subgradient} we have
$$\rho_F(y^k-a^i) \geq\frac{\|y^k - a^i\|}{\|F\|}\text{ and 
}a\geq \rho_F(a^i)\geq \dfrac{\|a^i\|}{\|F\|}.$$ This, together with~\eqref{key35}, implies that
\begin{equation}\label{mauthuan1}
\begin{array}{lll}
 \rho_F(y^k-a^i) \geq \dfrac{\|y^k - a^i\|}{\|F\|} &\geq \dfrac{\|y^k\| - \|a^i\|}{\|F\|}& \\
	&\geq (a + M + 1) - a\\
	& > M=f_F(y^1, \ldots, y^k)\\ 
     &= \sum_{p=1}^m\left(\min_{\ell =1,...,k} \rho_F(y^\ell - a^p)\right)\\
	& \geq  \min_{\ell =1,...,k} \rho_F(y^\ell - a^i). \end{array}	
\end{equation}
This clearly yields a contradiction if $k=1$. Consider the case where $k\geq 2$. It follows from~\eqref{mauthuan1} that $$\min_{\ell =1,...,k} \rho_F(y^\ell - a^i)= \min_{\ell =1,...,k-1} \rho_F(y^\ell - a^i) \;  \textrm{ for all }i=1, \ldots, m,$$ and hence we have
\begin{equation}\label{mauthuan}
    \sum_{i=1}^m \left(\min_{\ell =1,...,k} \rho_F(y^\ell - a^i)\right)= \sum_{i=1}^m\left(\min_{\ell =1,...,k-1} \rho_F(y^\ell - a^i)\right).
\end{equation} 
Since $k-1<m$, we can also assume without the loss of generality that $a^m\neq~y^{\ell}$ for all $\ell=1,...,k-1$.  Then we have 
$$\min_{\ell =1,...,k-1} \rho_F(y^\ell - a^m)=\rho_F(y^{k_1}-a^m)> 0 \textrm{ for some } k_1\in \{1,...,k-1\}.$$
Combining this with~\eqref{mauthuan}, we get
\begin{equation*}\label{eqnew}
\begin{aligned}
   f_F(y^1,...,y^k)&=\sum_{i=1}^{m-1}\left(\min_{\ell =1,...,k-1} \rho_F(y^\ell - a^i)\right)+\rho_F(y^{k_1}-a^m)  \\
     & >  \sum_{i=1}^{m-1}\left(\min_{\ell =1,...,k-1} \rho_F(y^\ell - a^i)\right)
     \\
     &\geq \sum_{i=1}^{m-1}\left(\min\left[\min_{\ell =1,...,k-1}\rho_F(y^\ell - a^i), \rho_F(a^m-a^i)\right]\right) \\
     &= \sum_{i=1}^{m}\left(\min\left[\min_{\ell =1,...,k-1}\rho_F(y^\ell - a^i), \rho_F(a^m-a^i)\right]\right) \\
     &=f_F(y^1,...,y^{k-1},a^m). 
\end{aligned} 
\end{equation*}
This contradicts to the assumption that $(y^1,...,y^k)\in S_F$, thereby completing the proof of the theorem.
\end{proof}

Following \cite{cuong2020qualitative, LNSYkcenter}, we inductively construct~$k$ subsets $A^1, \ldots, A^k$ from the set $A$ of the demand points in the following way: first put $A^0=\emptyset$ and then define
$$
	A^\ell=\left\{a^i \in A \, \backslash \, \left(\bigcup\limits_{q=0}^{\ell-1} A^q\right) \; \Big | \;  \rho_F(x^\ell-a^i)=\min _{r=1,\ldots,k}\rho_F(x^r-a^i)\right\}$$
 for $\ell = 1, \ldots , k$. The family $\{A^1,\ldots, A^k\}$   constructed  this way is called the \textit{natural clustering} with respect to $x=(x_1, \ldots, x_k)\in \R^{nk}$.
\begin{Definition}
	We say  that the component $x^\ell$ of $x=(x^1,...,x^k)\in \mathbb{R}^{nk}$ is {\em attractive} with respect to $A$ if the set
	\begin{equation}
		\label{attraction}
		A[x^\ell]=\left\{a^i \in A \; \Big |\; \rho_F(x^\ell-a^i)=\min_{r=1,\ldots,k}\rho_F(x^r-a^i)\right\}
	\end{equation}
	is nonempty.
\end{Definition}
\noindent It is easy to see that
$$A^\ell=A[ x^\ell]\setminus\left(\bigcup_{q=1}^{\ell-1}A^q\right).$$

The following proposition provides some properties of the global optimal solutions to~\eqref{maxminn}.

\begin{proposition}
	\label{pro3.5} Assume that $\bar{x}=(\bar{x}^1, \ldots, \bar{x}^k)$ is a global optimal solution to problem~\eqref{maxminn}. Then, the following assertions hold:
	\begin{enumerate}
		\item [{\rm (a)}] If $\left\{A^1, \ldots, A^k\right\}$ is the natural clustering  with respect to $\bar{x}=(\bar{x}^1, \ldots, \bar{x}^k)$, then $A^\ell$ is nonempty for every $\ell=1, \ldots, k$.
		\item [{\rm (b)}] The components of $\bar{x}$ are pairwise distinct, i.e., $\bar{x}^{\ell_1} \neq \bar{x}^{\ell_2}$ whenever $\ell_1 \neq~\ell_2$.
	\end{enumerate}
\end{proposition}
\begin{proof}
     (a)  We prove by contradiction. On the contrary, assume without loss generality that that $A^{1}=\emptyset$. Moreover, since $k\leq m$, we also can assume that $A^{2}$ contains at least two distinct points. This implies that there exists an $a^{i_1} \in A^{2}$ such that $a^{i_1} \neq \bar{x}^{2}$. Setting $\hat{x}^{1}:=a^{i_1}$ and $\hat{x}^\ell:=\bar{x}^\ell$ for all $\ell =2,\ldots,k$, we obtain
$$
\begin{small}
    \begin{array}{ll}
f_F(\hat{x}) & =\sum\limits_{a^i\in A^2}\left(\min\limits_{\ell=1,...,k} \rho_F(\hat x^\ell-a^i)\right)+\cdots
  +\sum\limits_{a^i\in A^k}\left(\min\limits_{\ell=1,...,k} \rho_F(\hat x^\ell-a^i)\right) \\[0.2in]
  &=\sum\limits_{a^i\in A^2}\left(\min\left[\min\limits_{\ell=2,...,k} \rho_F(\bar x^\ell-a^i),\rho_F(a^{i_1}-a^i)\right]\right)+\cdots\\[0.2in]
 &\textbf{} \quad\quad\quad\quad\quad\quad\quad +\sum\limits_{a^i\in A^k}\left(\min\left[\min\limits_{\ell=2,...,k} \rho_F(\bar x^\ell-a^i),\rho_F(a^{i_1}-a^i)\right]\right) \\[0.2in]
&=\sum\limits_{a^i\in A^2\setminus\{a^{i_1}\}}\min\left[\rho_F(\bar x^2-a^i),\rho_F(a^{i_1}-a^i)\right]
\\[0.2in]
&\textbf{} \quad\quad\quad\quad\quad\quad\quad +\sum\limits_{a^i\in A^3}\min\left[\rho_F(\bar x^3-a^i),\rho_F(a^{i_1}-a^i)\right] \\[0.2in]
& \textbf{} \quad\quad\quad\quad\quad\quad\quad + \cdots+\sum\limits_{a^i\in A^k}\min\left[\rho_F(\bar x^k-a^i),\rho_F(a^{i_1}-a^i)\right]
\\[0.2in]
&\leq\sum\limits_{a^i\in A^2\setminus\{a^{i_1}\}}\rho_F(\bar x^2-a^i)+\sum\limits_{a^i\in A^3}\rho_F(\bar x^3-a^i)+\cdots +\sum\limits_{a^i\in A^k}\rho_F(\bar x^k-a^i).
\end{array}
\end{small}$$
We also have 
\begin{small}
$$f_F(\bar x)= \rho_F(\bar x^2- a^{i_1})\; +\sum\limits_{a^i\in A^2\setminus\{a^{i_1}\}} \rho_F(\bar x^2-a^i) +\sum\limits_{a^i\in A^3} \rho_F(\bar x^3-a^i)+\cdots
  +\sum\limits_{a^i\in A^k} \rho_F(\bar x^k-a^i),$$
which gives us the  estimate
$$ f_F(\hat x)-f_F(\bar x) \leq-\rho_F(\bar{x}^{2}-a^{i_1})<0.$$
\end{small}
This yields a contradiction since $\bar{x}$ is a global optimal solution to problem (\ref{maxminn}).

(b) Suppose on the contrary that there exist distinct indexes $\ell_1, \ell_2 \in \{1,\ldots,k\}$ satisfying $\bar{x}^{\ell_1}=\bar{x}^{\ell_2}$. Since $k \leq m$, we have $k-1<m$, and thus there exists $\ell_0 \in \{1,\ldots,k\}$ such that $\left|A[\bar{x}^{\ell_0}]\right| \geq 2$ where $\left|A[\bar{x}^{\ell_0}]\right|$ denotes the number of elements of $A[\bar{x}^{\ell_0}]$. Therefore, one can find a point $a^{i_0} \in A[\bar{x}^{\ell_0}]$ such that $\bar{x}^{\ell_0} \neq a^{i_0}$. Set $\hat{x}^\ell=\bar{x}^\ell$ for every $\ell \in \{1,\ldots,k\} \backslash\left\{\ell_2\right\}$ and $\hat{x}^{\ell_2}=a^{i_0}$. Then, similar to the proof of part (a) we have
	$$
		f_F(\hat{x}^1, \ldots, \hat{x}^k)-f_F\left(\bar{x}^1, \ldots, \bar{x}^k\right) \leq-\rho_F(\bar{x}^{\ell_0}-a^{i_0})<0,
	$$
	which is a contradiction. The proof is now complete.
\end{proof}

%%%%%%%%%%%%%%%%%%%%%%%%%%%%%%%%%%%%%%

\subsection{Local Optimal Solutions}\label{subseclocal}

Recall that  a vector $\bar x=(\bar {x}^1, \ldots, \bar {x}^k) \in  \mathbb{R}^{nk}$ is said to be a {\em local optimal solution} to problem~(\ref{maxminn})
	if there exists $\epsilon >0$ such that for all $x = (x^1, \ldots, x^k) \in  \mathbb{R}^{nk}$ satisfying
	$\|x^\ell - \bar x^\ell\| < \epsilon$ for all $\ell=1,...,k$, we have $$f_F(\bar x) \leq  f_F(x).$$
	The set of all local optimal solutions to problem (\ref{maxminn}) is denoted by $S^{loc}_F.$

To establish necessary and sufficient conditions for the existence of local optimal solutions to problem (\ref{maxminn}), we find a DC decomposition of the  objective function $f_F$. Fix any $x=(x^1,\ldots,x^k)\in \mathbb{R}^{nk}$. Observe that
$$
	\begin{array} {lll}
		f_F(x) & = & \displaystyle\sum_{i=1}^m\min_{\ell=1,\ldots,k}\rho_F(x^\ell-a^i)                                                                  \\
		       & = & \displaystyle\sum_{i=1}^m\left(\sum_{r=1}^k\rho_F(x^r-a^i)- \max_{\ell=1,\ldots,k} \sum_{r=1,r \neq \ell}^k\rho_F(x^r-a^i)\right).
	\end{array}
$$
We first define
\begin{equation}
	\label{key1}	g(x)=\sum_{i=1}^m\left(\sum_{r=1}^k\rho_F(x^r-a^i)\right)
\end{equation}
and
\begin{equation}
	\label{key2}
	h(x)=\sum_{i=1}^m\left(\max_{\ell=1,\ldots,k} \sum_{r=1, r\neq \ell}^k\rho_F(x^r-a^i)\right).
\end{equation}
Then, we obtain a DC decomposition for the objective function of problem~(\ref{maxminn}):
\begin{equation}
	\label{key3}
	f_F(x)=g(x)-h(x), \; \ x\in \R^{nk}.
\end{equation}
Furthermore, we define $$h_{i,\ell}(x) = \displaystyle\sum_{r=1, r\neq\ell}^k \rho_F(x^r-a^i) \;  \ \text{for} \; \ (i,\ell)\in \{1,\ldots,m\}\times \{1,\ldots,k\},$$
and then set
\begin{equation}
	\label{key7}	h_i(x)=\max_{\ell=1,\ldots,k}h_{i,\ell}(x) \;  \text{ for}\;  i=1, \ldots, m.
\end{equation}
It follows that
\begin{equation}
	\label{key13}
	h(x)=\sum_{i=1}^m h_{i}(x).
\end{equation}
\begin{comment}
By applying the necessary optimality condition in DC programming (see, e.g.,~\cite[the proof of the ``Necessity'' assertion in Theorem 4.2]{van2016problem} and the references therein) to problem~(\ref{maxminn}), we obtain the following proposition.
\begin{proposition}
	\label{localsolution} If $\bar x =~(\bar {x}^1, \ldots, \bar {x}^k) \in  \mathbb{R}^{nk}$ is a local optimal solution to problem \eqref{maxminn}, then
	\begin{equation}
		\label{key4}
		\partial h(\bar x)\subset \partial g(\bar x),
	\end{equation}
	where $g$ and $h$ are defined by (\ref{key1}) and (\ref{key2}), respectively.
\end{proposition}
\end{comment}

Before presenting the main result of this subsection, we need the following proposition.
\begin{proposition}
Let $x=(x_1, \ldots, x_k)\in \R^{nk}$.	For every $i=1, \ldots, m$,  define
	\begin{equation}
		\label{key6}
		L_i(x) =\left\{ \ell=1,\ldots,k \;  \big | \;  h_{i,\ell}(x)=h_i(x) \right\}.
	\end{equation}
	Then
	\begin{equation}
		\label{key5}
		L_i(x) =\left\{ \ell=1, \ldots, k \;  \big | \;   a^i\in A[ x^\ell]  \right\} .
	\end{equation}
\end{proposition}
\begin{proof}
	By the definition of the function $h_{i,\ell}$, we have
	$$h_{i,\ell}= \displaystyle\sum_{r=1}^k\rho_F(x^r-a^i)-\rho_F(x^\ell-a^i).$$
	Then, it follows from (\ref{key7}) that
	\begin{equation}
		\label{f11}
		\begin{array}{ll}
			h_i(x) & =\displaystyle \max\limits_{\ell=1,\ldots,k} \left(\sum_{r=1}^k\rho_F(x^r-a^i)-\rho_F(x^\ell-a^i)\right)  \\
			       & =\displaystyle\sum_{r=1}^k\rho_F(x^r-a^i)+ \max\limits_{\ell=1,\ldots,k}\left (-\rho_F(x^\ell-a^i)\right) \\
			       & =\displaystyle\sum_{r=1}^k\rho_F(x^r-a^i)- \min\limits_{\ell=1,\ldots,k}\rho_F(x^\ell-a^i).
		\end{array}
	\end{equation}
	Thus, the maximum in (\ref{key7}) is attained at index $\ell$ if and only if the minimum in \eqref{f11} is
	attained at the same index. This, combined with (\ref{attraction}) and (\ref{key6}), shows that  (\ref{key5}) holds.
\end{proof}
Let $\bar{x}=(\bar{x}^1, \ldots, \bar{x}^k)\in \mathbb{R}^{nk}$. Suppose that $A[\bar{x}^\ell]$ is nonempty for some $\ell\in\{1,...,k\}$. Consider the index set 
\begin{equation}\label{Iell}
I(\bar x^\ell)=\{i\in I\; |\; a_i\in A[\bar{x}^\ell]\}    
\end{equation}
 and the associated single-source Weber problem:
\begin{eqnarray}
	\label{sWb} \min \left\{\ph^\ell_F(x) = \sum_{i\in I(\bar x^\ell)} \rho_F(x - a^i)\;  \Big{|}\; x\in \R^n \right\}.
\end{eqnarray}
The following lemma is instrumental in establishing both necessary and sufficient conditions for the existence of a local optimal solution.  
\begin{lemma}\label{lemmamoi}
Let $\bar{x} = (\bar{x}^1, \ldots, \bar{x}^k) \in \mathbb{R}^{nk}$. Suppose that for every $i=1, \ldots, m$, the index set $L_i(\bar{x})$ given by \eqref{key5} is a singleton. Then, there exists $\varepsilon > 0$ such that for any $x = \left(x^1, \ldots, x^k\right) \in \mathbb{R}^{nk}$ satisfying $\|x^r - \bar{x}^r\| < \varepsilon$ for all $r =1, \ldots, k$, the following properties are satisfied:
\begin{enumerate} 
		\item [{\rm (a)}] $L_i(x) = L_i(\bar{x})\ \; \mbox{\rm for all } i =1, \ldots, m.$
\item [{\rm (b)}] $I(\bar x^r)=I(x^r)$ and $A[x^r] = A[\bar x^r]$.
\end{enumerate}
\end{lemma}
\begin{proof}
For every $i=1,...,m$, let $\ell(i)$ denote the unique element of $L_i(\bar{x})$. Then, we obtain  from the definition of $L_i(\bar x)$ that 
     \begin{eqnarray}\label{25'}
     \rho_F(\bar{x}^{\ell(i)}-a^i)<\rho_F(\bar{x}^{\ell}-a^i)\; \text{ for all }\ell \in \{1,\ldots,k\} \setminus\{\ell(i)\}.    
     \end{eqnarray}
By the continuity of the Minkowski function,  for every $i=1,\ldots,m$ there exists $\varepsilon_i>0$ such that the following inequality holds:
	\begin{equation}
		\label{key25}
		\rho_F(x^{\ell(i)}-a^i)<\rho_F(x^{\ell}-a^i)\;  \text{ for all  } \ell \in \{1,\ldots,k\} \backslash\{\ell(i)\},
	\end{equation}
	whenever $x=\left(x^1, \ldots, x^k\right) \in \mathbb{R}^{n k}$ satisfies $\|x^r-\bar{x}^r\|<\varepsilon_i$ for all $r=1,\ldots,k.$
 For each such element $x$, it follows from~\eqref{key25} that $L_i(x)=\{\ell_i\}=L_i(\bar x)$. 
 
 Let $\varepsilon=\min\{\varepsilon_1, \ldots, \varepsilon_m\}$. Fix any $x=(x^1, \ldots, x^k)$ such that $\|\bar x^r-x^r\|<\epsilon$ for all $r=1, \ldots, k$. Then, we see that (a) is satisfied. 
 
 Now, we will prove (b). Fix any $r=1, \ldots, k$ and any $i\in I(\bar x^r)$. By definition, $a^i\in A[\bar x^r]$, and hence $r\in L_i(\bar x)=L_i(x)$. This implies that $a^i\in A[x^r]$, or equivalently, $i\in I(x^r)$. Thus, $I(\bar x^r)\subset I(x^r)$. Since the reverse inclusion is also obvious, we see that $I(\bar x^r)=I(x^r)$, and hence $A(\bar x^r)=A(x^r)$. The proof is now complete. 
 \end{proof}
 
 \begin{comment}
 if $A[\bar x^\ell]\neq \emptyset$, then it follows from~\eqref{key5} that $\ell=\ell(i)$  for all $i\in I(\bar x^\ell)$, where $I(\bar x^\ell)$ is given by~\eqref{Iell}. Since $L_i(x)=L_i(\bar x)=\{\ell(i)\}$, we get from~\eqref{25'}, \eqref{key25} and \eqref{key5} that
 $$A[x^\ell]=A[\bar x^\ell]\text{ for all }x^\ell\textrm{ satisfying }\|x^\ell-\bar{x}^\ell\|<\varepsilon.$$
 Consequently, $I(x^\ell)=I(\bar x^\ell)$. 
 Setting $\varepsilon=\min\{\varepsilon_1,...,\varepsilon_m\}$, we arrive at the conclusions of the lemma.
\end{proof}
\end{comment}

Now, we present necessary and sufficient conditions for the existence of a local optimal solution to problem~\eqref{maxminn}.

\begin{theorem}
	\label{theonec}
	Let $\bar{x}=(\bar{x}^1, \ldots, \bar{x}^k)\in \mathbb{R}^{nk}$. Suppose that the index set $L_i(\bar{x})$ is a singleton for every $i=1, \ldots, m$. Then, the following statements are equivalent:
	\begin{enumerate} 
		\item [{\rm (a)}] $\bar{x}=(\bar{x}^1, \ldots, \bar{x}^k)$ is a local optimal solution to problem (\ref{maxminn}).
		\item [{\rm (b)}] For every $ \ell=1,\ldots,k$, if $A[\bar{x}^\ell]$ is nonempty, then $\bar{x}^\ell$ is a (global) optimal solution of the generalized single-source
		      Weber problem~\eqref{sWb} defined by the data set $A[\bar{x}^\ell]$.
	\end{enumerate}
\end{theorem}

\begin{proof}
	(a) $\Longrightarrow$ (b): Suppose that  $A[\bar{x}^{ \ell}] \neq \emptyset$ for some $\ell \in \{1,...,k\}$. Since $\bar x$ is a local optimal solution to problem~\ref{maxminn}, by Lemma~\ref{lemmamoi} there exists $\varepsilon >0$ such that
 \begin{equation}\label{localsolution}
      A[\bar{x}^{r}] = A[x^r], \;  I(\bar x^r)=I(x^r), \; \mbox{\rm and} \; 
     f_F(\bar x)\leq f_F(x),
 \end{equation}
 whenever $x=(x^1, \ldots, x^k)$ satisfies $\|x^r-\bar x^r\|<\epsilon$ for all $r=1, \ldots, k$. Take any $x^\ell\in \R^n$ such that $\|x^\ell-\bar{x}^\ell\|<\epsilon$.
 Let $\hat{x}=\left(\hat{x}^1, \ldots, \hat{x}^k\right)\in \mathbb {R}^{nk}$ be the vector given by $\hat{x}^\ell={x}^\ell$ and $\hat{x}^r=\bar{x}^r$ for all $r \in \{1,\ldots,k\} \setminus\{\ell\}$. Then, it follows from~\eqref{localsolution} that 
$$\begin{array}{ll}
  f_F(\bar x)&=\sum\limits_{i\in I(\bar x^\ell)}\left(\min\limits_{r=1,...,k}\rho_F(\bar x^r-a^i)\right)+\sum\limits_{i\notin I(\bar x^\ell)}\left(\min\limits_{r=1,...,k}\rho_F(\bar x^r-a^i)\right)  \\[0.2in]
     &=\sum\limits_{i\in I(\bar x^\ell)}\rho_F(\bar x^\ell-a^i)+\sum\limits_{i\notin I(\bar x^\ell)}\left(\min\limits_{r=1, \ldots,k}\rho_F(\bar x^r-a^i)\right)\\[0.2in]
     &\leq f_F(\hat x)\\
     &=\sum\limits_{i\in I(\bar x^\ell)}\left(\min\limits_{r=1, \ldots,k}\rho_F(\hat x^r-a^i)\right)+\sum\limits_{i\notin I(\bar x^\ell)}\left(\min\limits_{r=1, \ldots,k}\rho_F(\hat x^r-a^i)\right)
     \\[0.2in]
     &=\sum\limits_{i\in I( \hat x^\ell)}\left(\min\limits_{r=1, \ldots,k}\rho_F( \hat x^r-a^i)\right)+\sum\limits_{i\notin I(\bar x^\ell)}\left(\min\limits_{r=1, \ldots,k}\rho_F(\bar x^r-a^i)\right)\\[0.2in]
     &=\sum\limits_{i\in I(\hat x^\ell)}\rho_F( \hat x^\ell-a^i)+\sum\limits_{i\notin I(\bar x^\ell)}\left(\min\limits_{r\in\{1,...,k\}}\rho_F(\bar x^r-a^i)\right)\\[0.2in]
     &=\sum\limits_{i\in I(\bar x^\ell)}\rho_F( x^\ell-a^i)+\sum\limits_{i\notin I(\bar x^\ell)}\left(\min\limits_{r=1, \ldots,k}\rho_F(\bar x^r-a^i)\right).
\end{array}$$
It follows that $$\sum\limits\limits_{i\in I(\bar x^\ell)}\rho_F(\bar{x}^{ \ell}-a^i)\leq\sum\limits_{i\in I(\bar x^\ell)} \rho_F({x}^{\ell}-a^i) .$$
Therefore, $\bar x^{\ell}$ is local optimal solution to the generalized single-source
	Weber problem \eqref{sWb} defined by the data set $A[\bar{x}^{ \ell}]$. Since the objective function of this problem is convex, $\bar x^{\ell}$ is global optimal solution to the problem. This completes the proof of (a) $\Longrightarrow$ (b).

(b) $\Longrightarrow$ (a):
	For every $i=1,...,m$, let $\ell(i)$ denote the unique element of $L_i(\bar{x})$. By Lemma~\ref{lemmamoi}, there exists $\varepsilon>0$ such that for any $x=\left(x^1, \ldots, x^k\right) \in \mathbb{R}^{n k}$ satisfying $\|x^r-\bar{x}^r\|<\varepsilon$ for all $r =1,\ldots,k,$
	we have $L_i(x)=\{\ell(i)\}$ for every $i=1,\ldots,m$. This  implies that
	\begin{equation} \label{f1}
	    \rho_F(x^{\ell(i)}-a^i)=\min _{\ell=1,\ldots,k }\rho_F(x^\ell-a^i).
	\end{equation}
Furthermore, it follows from condition (b) that
	\begin{equation}\label{f2}
	    		\sum_{i \in I(\bar x^\ell)} \rho_F(\bar{x}^\ell-a^i) \leq \sum_{i \in I(\bar x^\ell)} \rho_F(x^\ell-a^i)
	\end{equation}
	for every $\ell=1,\ldots,k$ satisfying $A[\bar{x}^\ell] \neq \emptyset$. Hence, we obtain that
	$$
		\begin{aligned}
			f_F(x)=\sum_{i =1}^m\min _{\ell=1,\ldots,k}\rho_F(x^\ell-a^i)
			 & =\sum_{i =1}^m\rho_F(x^{\ell(i)}-a^i)\quad (\textrm{ by \eqref{f1}}) \\
			 & =\sum_{\ell =1}^k\left(\sum_{i \in I(\bar x^\ell)} \rho_F(x^{\ell(i)}-a^i)\right)             \\
			 & =\sum_{\ell =1}^k\left(\sum_{i \in I(\bar x^\ell)} \rho_F(x^{\ell}-a^i)\right)               \\
			 & \geq \sum_{\ell =1}^k\left(\sum_{i \in I(\bar x^\ell)} \rho_F(\bar{x}^{\ell(i)}-a^i)\right) \quad (\textrm{ by \eqref{f2}})\\
              &=f_F(\bar{x}),
		\end{aligned}
	$$
	which yields  $\bar{x}=\left(\bar{x}^1, \ldots, \bar{x}^k\right)\in S_F^{loc}$. Note that in the proof above we use the obvious fact that $\{1,...,m\}=\cup_{\ell=1}^k I[\bar x_\ell]$, as a disjoint union, and also the convention that $\sum_{i\in \emptyset}=0$. 
The proof is now complete.
\end{proof}

\begin{example}
    Consider problem (\ref{maxminn}) with $\rho_F(x) = \|x\|$ and 4 data points in $\R^2$ located at the vertices of the unit square: $a^1 = (0,0), \; a^2 = (1,0) \; a^3 = (0,1) \; a^4 = (1,1)$. We seek two centers $x^1$ and $x^2$. Observe that $\bar x = (\bar x^1 = (0.5,0), \bar x^2 = (0.5,1))$ is a local optimal solution to (\ref{maxminn}). Indeed, we find that $A[\bar x^1] =\{a^1,a^2\} \ne \emptyset$ and $A[\bar x^2] =\{a^3,a^4\} \ne \emptyset$. Then, a direct verification shows that $\bar x^1$ ($\bar x^2$, respectively) is a global optimal solution of the generalized single-source
		      Weber problem~\eqref{sWb} defined by the data set $A[\bar{x}^1]$ ($A[\bar{x}^2]$, respectively).   Furthermore, $L_1(\bar{x}) = L_2(\bar{x}) = \{1\}$ and $L_3(\bar{x}) = L_4(\bar{x}) =~\{2\}$, meaning that $L_i(\bar{x})$ is a singleton for $i=1,2$. Therefore, by Theorem \ref{theonec}, $\bar{x}$ is a local solution to problem (\ref{maxminn}), and the local optimal value is 2. However, $f_F((0.2,0.2),a^4) \approx 1.932 < 2$, which shows that $\bar{x}$ is not a global solution. 
    
\end{example}

%%%%%%%%%%%%%%%%%%%%%%%%%%%%%%
%%%%%%%%%%%%%%%%%%%%%%%%%%%%%%%%
\section{The Boosted DCA for Generalized Multi-source Weber Problems} \label{sec: Multifacility Location}
In this section, we use the adaptive BDCA in \cite{Artacho2020} to solve the generalized multi-source Weber problem for both the unconstrained and constrained cases. These variations extend the DCA introduced in (\cite{Nam2017}) for the unconstrained case and in (\cite{Nam2018}) for the constrained case. For the sake of completeness, we first recall the relevant definitions and formulations before presenting the algorithms.

\subsection{Unconstrained Problems} \label{SEc41}

We begin by rewriting the objective function of GMWP as a difference of convex decomposition. In fact, as shown in Subsection~\ref{subseclocal}, the objective function of GMWP can be rewritten as

\[
	f_F (x^{1},\ldots,x^{k}) = \sum_{i =1}^{m } \sum_{\ell =1}^{k } \rho_{F }(x^{\ell }-a^{i}) - \sum_{i =1}^{m } \max_{r=1,\ldots,k } \sum_{\ell =1,\;\ell\neq r }^{k } \rho_{F }(x^{\ell }-a^{i}).
\]
Using \cite[Proposition 3.1 ]{Nam2017} and following the method in \cite{Nam2017}, we see that \(f_F \) can be approximated by the following function:
\begin{small}
\begin{align*}
	f_\mu(x^1, \ldots, x^k) & =\frac{1}{2\mu}\sum_{i=1}^m\sum_{\ell=1}^k\left\|x^\ell-a^i\right\|^2                                                                                                                                          \\
	                        & -\left[\frac{\mu}{2}\sum_{i=1}^m\sum_{\ell=1}^k\left[ d\left(\frac{x^\ell-a^i}{\mu}; F^{\circ}\right)  \right] ^2+\sum_{i=1}^m\max_{r=1, \ldots, k}\sum_{\ell=1,\; \ell\neq r}^k \rho_{F }\left(x^\ell-a^i\right)\right]
\end{align*}
\end{small}
\noindent where the parameter $\mu>0$ controls the trade-off between smoothness and approximation accuracy, and $d(x;F^\circ)$ denotes the Euclidean distance from a point $x$ to $F^\circ$. Let $X$ be the $k \times n$ matrix whose rows are $x^{1}, \ldots, x^{k}$. Now, we define the following functions:
$$
	\begin{aligned}
		G_{\mu}\left(X\right) & =\frac{1}{2 \mu} \sum_{i=1}^{m} \sum_{\ell=1}^{k}\left\|x^{\ell}-a^{i}\right\|^{2},  \\
		H_{\mu}\left(X\right) & =\frac{\mu}{2} \sum_{i=1}^{m} \sum_{\ell=1}^{k}\left[d\left(\frac{x^{\ell}-a^{i}}{\mu} ; F^{\circ}\right)\right]^{2}+\sum_{i=1}^{m} \max _{r=1, \ldots, k} \sum_{\ell=1, \ell \neq r}^{k} \rho_{F}\left(x^{\ell}-a^{i}\right).
	\end{aligned}
$$
It follows that
$$f_\mu(X) = G_{\mu}\left(X\right) - H_{\mu}\left(X\right).$$

Next, we revisit the process of finding the subdifferential of $f_\mu$ established in~\cite{Nam2017}. For this purpose, we consider the inner product space $\mathcal{M}$ of all $k \times n$ matrices, with the inner product of $A, B \in \mathcal{M}$ defined as
$$
	\langle A, B\rangle =\operatorname{trace}\left(A B^{T}\right)=\sum_{i=1}^{k} \sum_{j=1}^{n} a_{i j} b_{i j}.
$$

\noindent The Frobenius norm induced by this inner product is defined by
$$
	\|X\|_{\mathcal{F}}=\sqrt{\sum_{i=1}^{k} \sum_{j=1}^{n} x^2_{ij}}.
$$

\noindent Then, we observe that
$$
	\begin{aligned}
		G_{\mu}(X) & =\frac{1}{2 \mu} \sum_{\ell=1}^{k} \displaystyle\sum_{i=1}^{m}\left(\left\|x^{\ell}\right\|^{2}-2\left\langle x^{\ell}, a^{i}\right\rangle+\left\|a^{i}\right\|^{2}\right) \\
		           & =\frac{1}{2 \mu}\left(m\|X\|_{\mathcal{F}}^{2}-2\langle X, B\rangle+k\|A\|^{2}\right)                                                                                      \\
		           & =\frac{m}{2 \mu}\|X\|_{\mathcal{F}}^{2}-\frac{1}{\mu}\langle X, B\rangle+\frac{k}{2 \mu}\|A\|^{2},
	\end{aligned}
$$
where $A$ is the $m \times n$ matrix whose rows are $a^{1}, \ldots, a^{m}$, and $B$ is the $k \times n$ matrix with $a=\sum_{i=1}^{m} a^{i}$ for every row. The function $G_{\mu}$ is differentiable, with the gradient given by
$$
	\nabla G_{\mu}(X)=\frac{m}{\mu} X-\frac{1}{\mu} B.
$$
Since $G_\mu$ is continuous, $X=\nabla G_\mu^*(Y)$ if and only if $Y=\nabla G_{\mu}(X)$, and hence we have
$$\nabla G_\mu^*(Y)=\frac{1}{m}(B+\mu Y).$$
\noindent Next, we define
$$
	\begin{array}{lll}
		H_{\mu}^{1}(X) & =\dfrac{\mu}{2} \displaystyle\sum_{i=1}^{m} \sum_{\ell=1}^{k}\left[d\left(\frac{x^{\ell}-a^{i}}{\mu} ; F^{\circ}\right)\right]^{2}, \\
		H^{2}(X)       & = \displaystyle\sum_{i=1}^{m} \max _{r=1, \ldots, k} \sum_{\ell=1, \ell \neq r}^{k} \rho_{F}\left(x^{\ell}-a^{i}\right).
	\end{array}
$$
\noindent Then, $H_\mu(X) = H^1_\mu + H^2$. It follows directly from Proposition 3.1 in \cite{Nam2017} that

$$
	\begin{aligned}
		 & \frac{\partial H_{\mu}^{1}}{\partial x^{1}}(X)=\frac{x^{1}-a^{1}}{\mu}-P\left(\frac{x^{1}-a^{1}}{\mu} ; F^{\circ}\right)+\cdots+\frac{x^{1}-a^{m}}{\mu}-P\left(\frac{x^{1}-a^{m}}{\mu} ; F^{\circ}\right) \\
		 & \vdots                                                                                                                                                                                                    \\
		 & \frac{\partial H_{\mu}^{1}}{\partial x^{k}}(X)=\frac{x^{k}-a^{1}}{\mu}-P\left(\frac{x^{k}-a^{1}}{\mu} ; F^{\circ}\right)+\cdots+\frac{x^{k}-a^{m}}{\mu}-P\left(\frac{x^{k}-a^{m}}{\mu} ; F^{\circ}\right)
	\end{aligned}
$$
where $P(x;F^\circ)$ denotes the Euclidean projection of a point $x$ onto $F^\circ$. The gradient $\nabla H_{\mu}^{1}(X)$ is defined by
$$\nabla H_{\mu}^{1}(X) =
	\begin{pmatrix}
		\frac{\partial H_{\mu}^{1}}{\partial x^{1}}(X) \\
		\vdots                                          \\
		\frac{\partial H_{\mu}^{1}}{\partial x^{k}}(X)
	\end{pmatrix}
	_{k\times n}$$
In what follows, we outline a procedure  to find a subgradient of $H^{2}$ at $X$.\\
$\bullet$ Define the function
	      $$Q^{i, r}(X)=\sum_{\ell=1, \ell \neq r}^{k} \rho_{F}\left(x^{\ell}-a^{i}\right)\textrm{
			      and
		      }Q^{i}(X)=\max _{r=1, \ldots, k} Q^{i, r}(X).$$
	      By these definitions, $\partial H^2 (X) \subseteq \sum_{i=1}^m \partial Q^i (X) $. Therefore, to find $\partial H^2(X)$, we need to locate a particular subgradient $V_{i} \in \partial Q^{i}(X)$ for each $i=1,\ldots,m$. As a result, a subgradient of $H^2$ can be given by $\sum_{i=1}^m V_i$. \\
$\bullet$ Now, to find $V_i$ for each $i=1,\ldots, m$, we  proceed in two steps:
	      \begin{itemize}
		      \item[$\circ$] First, find an index $\ell^* \in \{1,\ldots,k\}$ such that
		            $$Q^{i}(X)=Q^{i, \ell^*}(X) \;  \text{for} \,\, i=1,\ldots, m.$$
		      \item[$\circ$] Next, for each $i=1,\ldots,m$, find $V_{i} \in \partial Q^{i}(X)$ as follows:
		        \begin{itemize}
			            \item The $\ell^{*th}$ row of $V_i$ denoted $v_i^{l^*}$ is zero;
			            \item For the $j^{th}$ row of $V_i$, where $j \ne \ell^*$, denoted $v_i^j$, it is given by
			                 
                      $$v_i^j =
				                  \begin{cases}
					                  \partial \rho_F(x^j - a^i) & \text{if} \;  x^j \ne a^i, \\
					                  0 &\text{if} \;  x^j = a^i.
				                  \end{cases}
			                  $$
		            \end{itemize}
	      \end{itemize}
\noindent With the necessary definitions and calculations established, we now present our adaptive BDCA for the Generalized Multi-Source Weber Problem, utilizing the technique from \cite{Nam2018} of gradually decreasing the smoothing parameter $\mu$.

\begin{algorithm}[H]
	\caption{Adaptive BDCA for solving (\ref{maxminn}) } \label{BDCA unconstrained}
	\begin{algorithmic}
		\Procedure{}{}
		\State Data: Given $a^1, \ldots, a^m \in \mathbb{R}^n$. Choose a reasonable number $k$, $1\le k\le m$.
		\State Choose $X_0\in \mathcal{M}, N\in  \mathbb{N}$, $\mu>\mu_{f} >0, \alpha >0, 0<\beta <1$, and $0<\delta <1$.
		\While{ $\mu>\mu_{f}$ }
		\For{$p=0,\ldots,N$}
		\State \underline{Step 1:} Find $Y_p = Y^1_p + Y^2_p$ where
        \vspace{-1em}
		$$\begin{array}{ll}
				Y^1_p & = \nabla H^1_\mu (X_p), \\
				Y^2_p & \in \partial H^2(X_p).
			\end{array} $$
        Find $Z_{p} = \nabla G^*_\mu(Y_p) =\frac{1}{m}\left(B+\mu Y_p\right)$.
		\State \underline{Step 2:} Set $d_p = Z_{p} - X_{p}$. If $d_p=0$ then Stop. %Otherwise, go to Step~3.
		\State \underline{Step 3:} Find an initial $\lambda_p$ using \cref{algorithm: self-adaptive step size}  %Choose any $\Bar{\lambda}_p \geq 0$, set $\lambda_p =\Bar{\lambda}_p$
		\While{$f_F(Z_{p} + \lambda_p d_p) > f_F (Z_{p}) - \alpha \lambda_p^2 \Vert d_p \Vert^2$}
		\State $\lambda_p = \beta \lambda_p$
		\EndWhile
		\State Set  $X_{p+1} = Z_{p} + \lambda_p d_p$. If $X_{p+1} = X_p$ then Stop.
		\EndFor
		\EndWhile
		\State Reassign $\mu = \delta \mu$
		\EndProcedure
		\State \Output{$X_{p+1}$}
	\end{algorithmic}
\end{algorithm}

Note that if $\bar{\lambda}_p =0$, the aBDCA is identical to the DCA introduced  in \cite{Nam2017}.

\subsection{Constrained Problems}\label{sec: MLP}

In this subsection, we consider a version of the generalized multi-source Weber problem with constraints, also known as the multi-facility location problem with constraints, as discussed in \cite{Nam2018}. We introduce an algorithm in \cite{Artacho2020} that is designed to reduce the number of iterations and computational time for solving these types of problems. Given a set of $m$ data-target points $a^1, a^2, \ldots, a^m$ in $\R^n$, our goal is to find $k$ centers $x^\ell$, each  belonging to constraint sets $\bigcap_{i=1}^q{\O_i^\ell}$ for $l=1,..,k$, such that the sum of the smallest Minkowski distances from these centers to the data-target points is minimized. We assume, without loss of generality, that the number of constraints is the same for each center. Then the problem can be defined as follows:
\begin{equation}
	\label{MF_problem}
	\begin{array}{ll}
		\min                   & f_F (x^{1}, \hdots, x^{k}) = \displaystyle\sum\limits_{i =1}^{m }  \left(\min\limits_{\ell=1,\ldots,k }\rho_{F }(x^{\ell }-a^{i})\right) \\
		[0.1in]
		\mbox{\rm subject to } & x^\ell\in\displaystyle\bigcap_{j=1}^q \O_j^\ell  \;  \text{ for} \, \ell=1, \ldots, k.
	\end{array}
\end{equation}
To tackle this problem, we proceed in two steps: first, we employ the penalty method to transform the constrained problem into an unconstrained one as shown in Section 6 of \cite{Nam2018}, and then we apply Nesterov's smoothing method from \cite{Nam2017} to approximate the objective function. For more details, the problem \eqref{MF_problem} is equivalent to the following problem:
\begin{small}
\begin{equation} \label{MF_problemunconstrained}
		\begin{array}{ll}
			\min  f_\tau (x^{1}, \hdots, x^{k}) & = \displaystyle\sum\limits_{i =1}^{m }  \left(\min\limits_{\ell=1,\hdots,k }\rho_{F }(x^{\ell }-a^{i})\right) + \frac{\tau}{2}\sum\limits_{\ell=1}^k\sum\limits_{j=1}^q\left[d(x^\ell; \Omega_{j}^\ell)\right]^2 \\
			[0.2in]
			\text{\rm subject to }              & x^\ell\in\mathbb{R}^n\; \text{\rm for }\ell=1, \ldots, k,
		\end{array}
\end{equation}
\end{small}
where $\tau>0$ is the penalty parameter. Now, by using Nesterov's smoothing technique from \cite{Nam2017} to approximate the objective function $f_\tau$, we obtain a new DC function that is suitable for using the DCA.

\begin{small}
	\begin{equation*}
		\begin{array}{ll}
			 & f_{\tau, \mu}\left(x^1, \ldots, x^k\right)= \left( \frac{1}{2\mu} \sum\limits_{i=1}^{m}\sum\limits_{\ell =1}^{k} \| x^\ell - a^i \|^2  + \frac{\tau q}{2} \sum\limits_{\ell=1}^k\|x^\ell\|^2 \right)- \\
			 & \left(  \frac{\mu}{2} \sum\limits_{i=1}^{m}\sum\limits_{\ell =1}^{k}\left[d\left(\frac{x^\ell - a^i}{\mu};F^\circ\right)\right]^2
			+   \sum\limits_{i=1}^m \max\limits_{r=1, \ldots,k}\sum\limits_{\ell=1, \ell\neq r}^k {\|x^\ell-a^i\|} \;
			+ \frac{\tau}{2} \sum\limits_{\ell=1}^k \sum\limits_{j=1}^q  \ph_{{\O}_j^{\ell}}\left(x^\ell\right) \right),

		\end{array}
	\end{equation*}
\end{small}

\noindent where $\mu$ is the smoothing parameter, and for every $\ell\in\{1,\ldots,k\}$ and every $j~\in \{1,\ldots,q\}$, the function $\ph_{\O^\ell_j}$ is defined as:
\begin{equation*}
	\ph_{\O^\ell_j}(x^\ell)=2\sup \left\{ \la  x^\ell, w \ra -\frac{1}{2}\|w\|^2 \; |\;    w\in \O^\ell_j \right\}.
\end{equation*}
The subgradients and subdifferentials of the DC decomposition are detailed in~\cite{Nam2018}. For ease of reference, we provide the relevant formulas for these quantities here without further explanation. The difference between the objective function $f_{\tau,\mu}$ in the constrained problem, and $f_\mu$ in the unconstrained problem arises from two terms involving $\tau$. Thus, we define the objective function as $f_{\tau,\mu} = G_{\tau,\mu} - H_{\tau,\mu}$, where $G_{\tau,\mu} = G_\mu + G_\tau$ and $H_{\tau,\mu} = H_\mu + H_\tau$. Here, $G_\mu, H_\mu$ are defined as in Section \ref{SEc41}, while $G_\tau = \dfrac{\tau q}{2} \sum\limits_{\ell=1}^k\|x^\ell\|^2$ and $H_\tau = \frac{\tau}{2} \sum_{\ell=1}^k \sum_{j=1}^q  \ph_{{\O}_j^{\ell}}(x^\ell)$. To minimize the objective function, we need to find the gradients of $G_\tau$ and $H_\tau$, which, as shown in \cite{Nam2018}, are $\nabla G_\tau = \tau q X$ and $\nabla H_\tau = \tau U$. Here, $X$ is the matrix defined in Section \ref{SEc41} and $U$ is the the matrix whose row is $\sum_{i=1}^q P(x^\ell; \Omega^\ell_i)$. Therefore, the gradients and subgradients are given by
$$\nabla G_{\tau,\mu} = \nabla G_{\mu} + \nabla G_{\tau}  \;  \text{and} \;  \partial H_{\tau,\mu} = \partial H_{\mu} + \nabla H_{\tau}. $$
%$$
%\begin{array}{ll}
%  \partial H_{\tau,\mu}  & = \partial H_{\mu} + \partial H_{\tau} \\
%   &  = \partial H_{\mu} + \nabla H_{\tau}.
%\end{array}
% $$
Using the same arguments as in Subsection \ref{sec: MLP}, we have
$$
	\nabla G_{\tau,\mu}^*(Y)=\dfrac{1}{m+\mu\tau q}(B+\mu Y).
$$
Using the technique from \cite{Nam2018} of gradually decreasing and increasing the values of the smoothing parameter $\mu$, and penalty parameter  $\tau$ respectively, we now present the adaptive BDCA for the constrained problem.
\begin{algorithm}[H]
	\caption{Adaptive BDCA for solving (\ref{MF_problem}) } \label{BDCA constrained}
	\begin{algorithmic}
		\Procedure{}{}
		\State Data: Given $a^1, \ldots, a^m \in \mathbb{R}^n$ and $\{ \O_j^\ell\}$ for $j=1,\ldots,q$ and $\ell=1,\ldots,k$.
		\State Choose $X_0\in \mathcal{M}, N\in  \mathbb{N}$, $\mu>\mu_{f} >0,  \tau_{f} > \tau >0, \alpha >0$, $0<\beta <1$, $\sigma>1$, and $0< \delta <1$.
		\While{ $\mu>\mu_{f} $ and $\tau < \tau_{f} $ }
		\For{$p=0,\ldots,N$}
		\State \underline{Step 1:} Find $Y_p = Y^1_p + Y^2_p + Y^3_p$ where
		$$
			\begin{array}{lll}
				Y^1_p & = \nabla H^1_\mu (X_p), \\
				Y^2_p & \in \partial H^2(X_p) ,   \\
				Y^3_p & =\tau U.
			\end{array}
		$$
		Find $Z_{p} = \dfrac{1}{m+\mu\tau q}\left(B+\mu Y_p\right)$.
		\State \underline{Step 2:} Set $d_p = Z_{p} - X_p$. If $d_p=0$ then Stop. %Otherwise, go to Step 3.
		\State \underline{Step 3:} Find an initial $\lambda_p$ using \cref{algorithm: self-adaptive step size}
		\While{$f_F(Z_{p} + \lambda_p d_p) > f_F (Z_{p}) - \alpha \lambda_p^2 \Vert d_p \Vert^2$}
		\State $\lambda_p = \beta \lambda_p$
		\EndWhile
		\State Set  $X_{p+1} = Z_{p} + \lambda_p d_p$. If $X_{p+1} = X_p$ then Stop.
		\EndFor
		\EndWhile
		\State Reassign $\tau=\sigma \tau$ and $\mu = \delta \mu$
		\EndProcedure
		\State \Output{$X_{p+1}$}
	\end{algorithmic}
\end{algorithm}
%%%%%%%%%%%%%%%%%%%%%%%%%%%%%%
%%%%%%%%%%%%%%%%%%%%%%%%%%%%%%%%
\section{Numerical Results}\label{Results}
In order to obtain practical benefits to our adaptive BDCA, we introduce a `skipping' step in Steps 2 and 3 of the Algorithms \ref{BDCA unconstrained} and \ref{BDCA constrained}, not present in the adaptive BDCA of \cite{Artacho2020}. This modification, as our numerical experiments show, is key to improving the generalized multi-source Weber problems by avoiding unnecessary line searches.

\begin{algorithm}[H]
	\caption{Adaptive BDCA for solving (\ref{maxminn}) with a `skipping' step  }
	\label{BDCA unconstrained - 2}
	\begin{algorithmic}
		\Procedure{}{}
		\State Data: Given $a^1, \ldots, a^m \in \mathbb{R}^n$. Choose a reasonable number $k$, $1\le k\le m$.
		\State Choose $X_0\in \mathcal{M}, N\in  \mathbb{N}$, $\mu> \mu_{f} >0, \lambda_0 > \lambda_{f}>0, \alpha >0$, $0<\beta <1$, and $0< \delta <1$.
		\While{ $\mu>\mu_{f}$ }
		\For{$p=0,\ldots,N$}
		\State \underline{Step 1:} Find $Y_p = Y^1_p + Y^2_p$ where
		$$
			\begin{array}{cc}
				Y^1_p & = \nabla H^1_\mu (X_p), \\
				Y^2_p & \in \partial H^2(X_p).
			\end{array}
		$$
		Find $Z_{p} = \frac{1}{m}\left(B+\mu Y_p\right)$.
		\If{$\|Z_{p}-X_{p}\|_\mathcal{F} > 10^{-6}$}
		\State \underline{Step 2:} Set $d_p = Z_{p} - X_{p}$.
		\If{$\lambda_{p} > \lambda_{f}$}
		\State \underline{Step 3:} Find an initial $\lambda_p$ using \cref{algorithm: self-adaptive step size}

		\While{$f_F(Z_{p} + \lambda_p d_p) > f_F (Z_{p}) - \alpha \lambda_p^2 \Vert d_p \Vert^2$}
		\State Reassign $\lambda_p = \beta \lambda_p$
		\EndWhile
		\State $X_{p+1} = Z_{p} + \lambda_p d_p$.
		\Else
		\State $X_{p+1} = Z_{p} $.
		\EndIf
		\Else
		\State Return $X_{p+1} = Z_{p} $
		\EndIf
		\EndFor
		\EndWhile
		\State Reassign $\mu = \delta \mu$
		\EndProcedure
		\State \Output{$X_{p+1}$}
	\end{algorithmic}
\end{algorithm}

\begin{algorithm}[H]
	\caption{Adaptive BDCA for solving (\ref{MF_problem}) with a `skipping' step} \label{BDCA constrained - 2}
	\begin{algorithmic}
		\Procedure{}{}
		\State Data: Given $a^1, \ldots, a^m \in \mathbb{R}^n$ and $\{ \O_j^\ell\}$ for $j=1,\ldots,q$ and $\ell=1,\ldots,k$.
		\State Choose $X_0\in \mathcal{M}, N\in  \mathbb{N}$, $\mu>\mu_{f} >0,  \tau_{f} > \tau >0, \lambda_0 > \lambda_{f}>0, \alpha >0$, $0<\beta <1$, $\sigma>1$, and $0< \delta <1$.
		\While{ $\mu>\mu_{f} $ and $\tau < \tau_{f} $ }
		\For{$p=0,\ldots,N$}
		\State \underline{Step 1:} Find $Y_p = Y^1_p + Y^2_p + Y^3_p$, where

		$$
			\begin{array}{lll}
				Y^1_p & = \nabla H^1_\mu (X_p), \\
				Y^2_p & \in \partial H^2(X_p),  \\
				Y^3_p & =\tau U.
			\end{array}
		$$
		Find $Z_{p} = \frac{1}{m+\mu\tau q}\left(B+\mu Y_p\right)$.
		\If{$\|Z_{p}-X_{p}\|_\mathcal{F} > 10^{-6}$}
		\State \underline{Step 2:} Set $d_p = Z_{p} - X_{p}$.
		\If{$\lambda_{p} > \lambda_{f}$}
		\State \underline{Step 3:} Find an initial $\lambda_p$ using \cref{algorithm: self-adaptive step size}

		\While{$f_F(Z_{p} + \lambda_p d_p) > f_F (Z_{p}) - \alpha \lambda_p^2 \Vert d_p \Vert^2$}
		\State Reassign $\lambda_p = \beta \lambda_p$
		\EndWhile
		\State $X_{p+1} = Z_{p} + \lambda_p d_p$.
		\Else
		\State $X_{p+1} = Z_{p} $.
		\EndIf
		\Else
		\State Return $X_{p+1} = Z_{p} $
		\EndIf
		\EndFor
		\EndWhile
		\State Reassign $\tau=\sigma \tau$ and $\mu = \delta \mu$
		\EndProcedure
		\State \Output{$X_{p+1}$}
	\end{algorithmic}
\end{algorithm}
We now implement and test  \cref{BDCA unconstrained,BDCA constrained,BDCA unconstrained - 2,BDCA constrained - 2} in MATLAB R2023a on an iMac (4.8 Ghz 8-Core i7, 32 GB DDR4). The parameters used are $\alpha = 0.05,\,\beta = 0.01,\gamma = 2, \tau=1,\mu=1, \tau_f=10^8, \mu_f=10^{-6}$ with the penalty growth parameter $\sigma$ for $\tau$, the smoothing decrement parameter $\delta$ for $\mu$, and $\lambda$ for the line search step detailed per example. All algorithms stop when $\|X_{p}-X_{p-1}\|_\mathcal{F}<10^{-6}$, then we update $\tau$ and $\mu$.
%%%%%%%%%%%%%%%%%%%%%%%%%%%%%%%%%%%%
%%%%%%%%%%%%%%%%%%%%%%%%%%%
\begin{example}\label{Triangle_Example} We test a simple unconstrained example solvable by hand to verify \cref{BDCA unconstrained}. We consider just 3 data points in $ \R^2$, with $a^1 = (0,0), \ a^2 = (1,0), \ a^3 = (0,1)$, and with $\rho_F(x) = \|x\|_2$. We seek 2 centers, $x^1$ and $x^2$. Two global solutions, up to permutation of the centers, are: 
			$$(x^1,x^2) = (a^2,a^3/2) \text{ and } (x^1 ,x^2) =(a^3,a^2/2),$$
	with an objective function value of $1$. Given the small problem size, we test \cref{BDCA unconstrained} using $\lambda_{\text{start}} = 1$, $\delta=0.8$,  and random initial values within the box $[0,0.5]^2$. The computed global solutions, visualized in \Cref{fig: triangle_example visual}, confirms that the algorithm converges correctly. 
\end{example}
\begin{figure}[htb!]
	\centering
\includegraphics[width=0.32\textwidth]{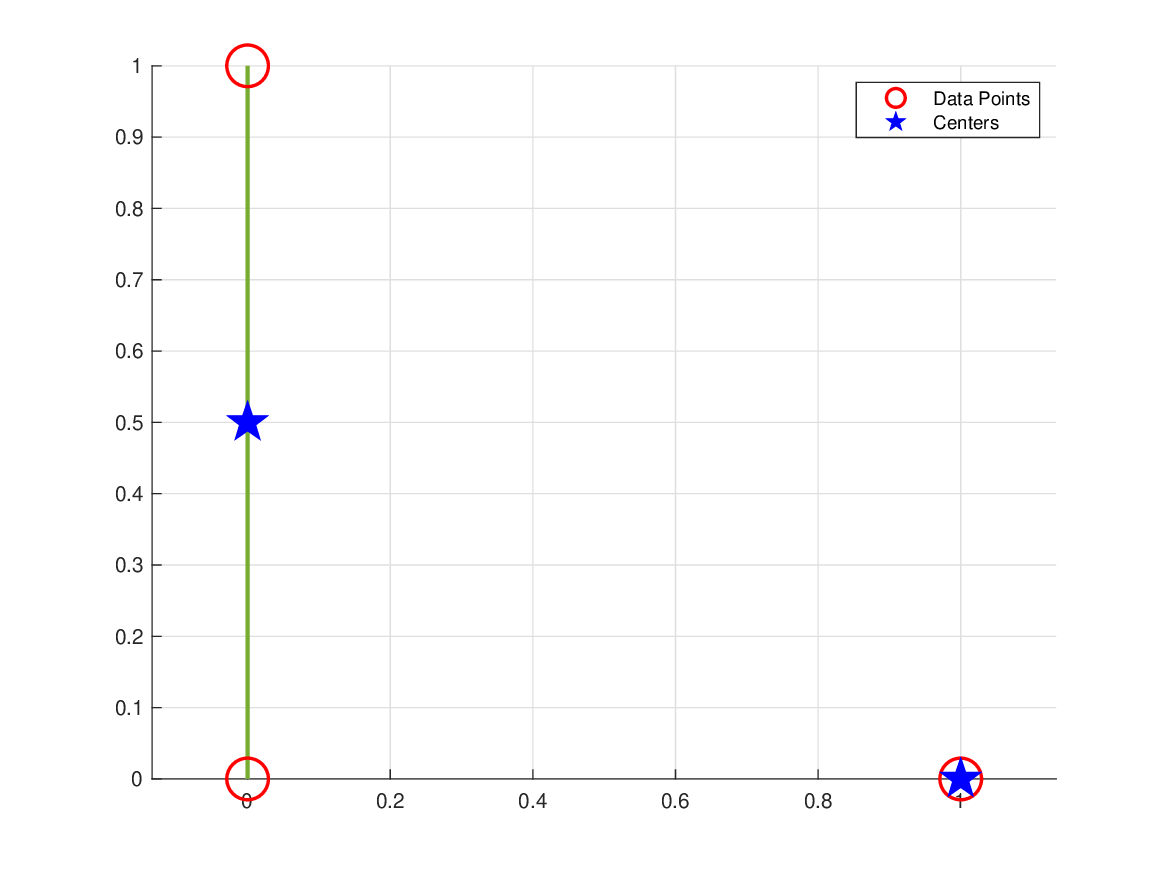}

 \caption{Visualization of the global solution to \Cref{Triangle_Example}.}
	\label{fig: triangle_example visual}
\end{figure}

\begin{example}\label{example: 4 points} We consider a simple unconstrained example to verify \Cref{BDCA unconstrained} using non-trivial Minkowski gauges and illustrate the differing global solutions they produce. We examine 4 data points in $\R^2$ on the unit square: $a^1 = (0,0), \; a^2 = (1,0) \; a^3 = (0,1) \; a^4 = (1,1)$, and seek two centers $x^1$ and $x^2$. The three gauges are considered: the usual $\rho_F(x) = \|x\|_2$, along with $\rho_F(x) = \|x\|_1$ and $\rho_F(x) = \|x\|_\infty$. A visualization of the computed global solutions using \Cref{BDCA unconstrained} is shown in \Cref{fig: 4 points}. \Cref{fig: 4points L2} corresponds to an exact global solution where one center is the Fermat-Torricelli point of the triangle formed by $a^1,a^2, a^4$, with other center at $a^3$. Permutations of the centers yield alternative global solutions. For $\rho_F(x) = \|x\|_1$, many global solutions exist, including $x^1 \in [a^3,a^4]$ and $x^2 \in [a^1,a^2]$ or $x^1 \in [a^1,a^4]$ and $x^2 \in [a^2,a^3]$. \Cref{fig: 4points L1} demonstrates \cref{BDCA unconstrained} converging to one of these exact global solutions. For $\rho_F(x) = \|x\|_\infty$, \cref{fig: 4points Linf} shows the exact global solution of $x^1 = (0.5,0.5)$ and $x^2 = a^1$, with other global solutions $x^2= a^i$ for $i=2,3,4$. In this example, we use $\lambda_{\text{start}} = 1$, $\delta=0.5$, and random initial values within the square. This example shows that Algorithm \ref{BDCA unconstrained} converges correctly for complex Minkowski gauges and demonstrates how different gauges can yield distinct solutions and clustering outcomes, even in simple cases.
    
    \begin{figure}[htb!]
        \centering
        \begin{subfigure}[b]{0.32\textwidth}
            \includegraphics[width=\textwidth]{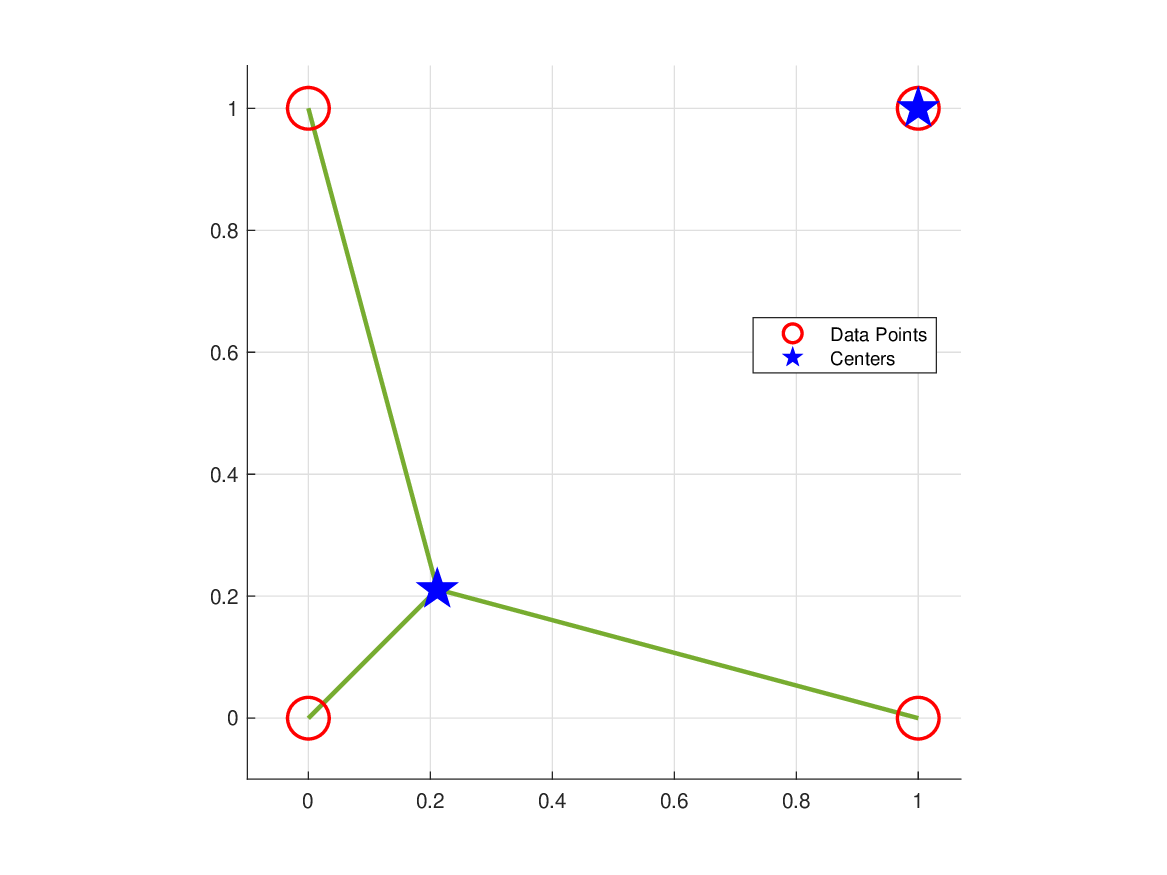}
            \caption{$\rho_F(x) = \|x\|_2$}
            \label{fig: 4points L2}
        \end{subfigure}
        \hfill
        \begin{subfigure}[b]{0.32\textwidth}
            \centering
            \includegraphics[width=\textwidth]{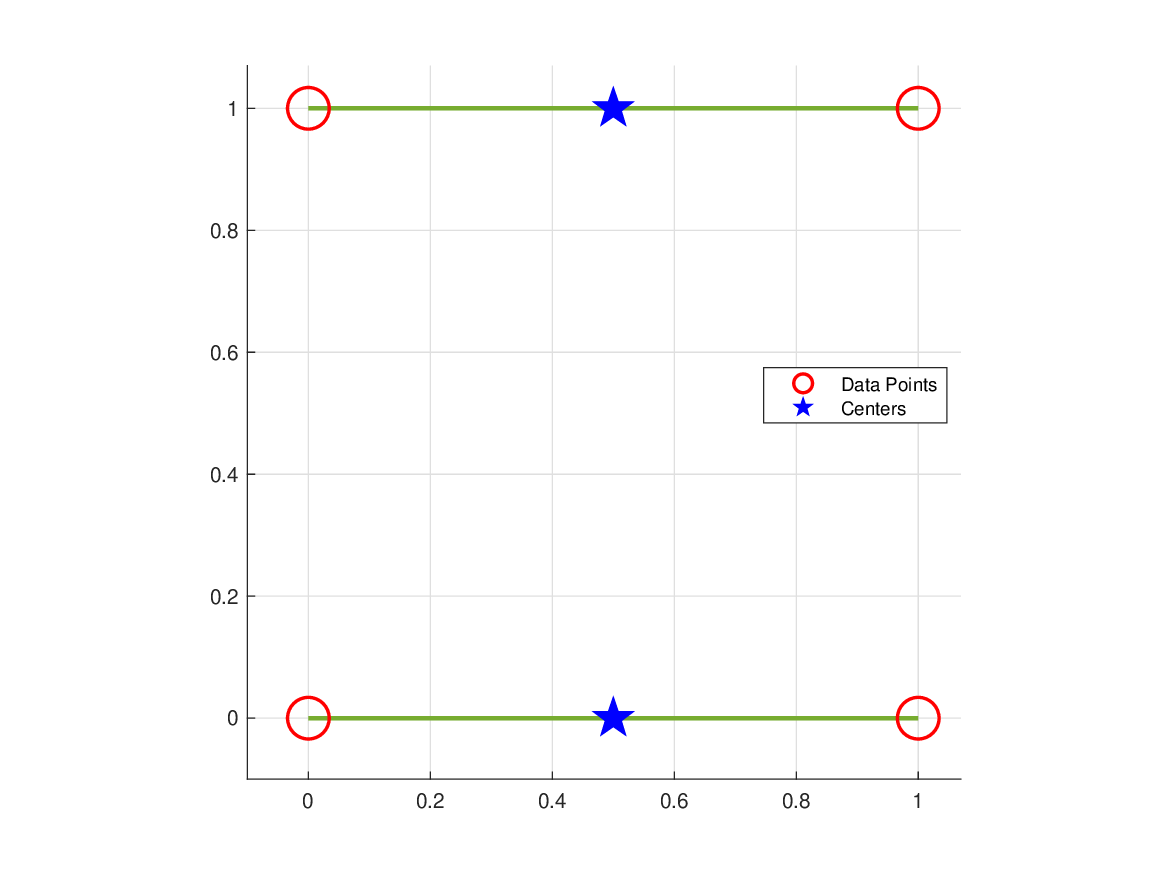}
            \caption{$\rho_F(x) = \|x\|_1$}
            \label{fig: 4points L1}
        \end{subfigure}
        %\vskip\baselineskip
        \begin{subfigure}[b]{0.32\textwidth}
            \centering
            \includegraphics[width=\textwidth]{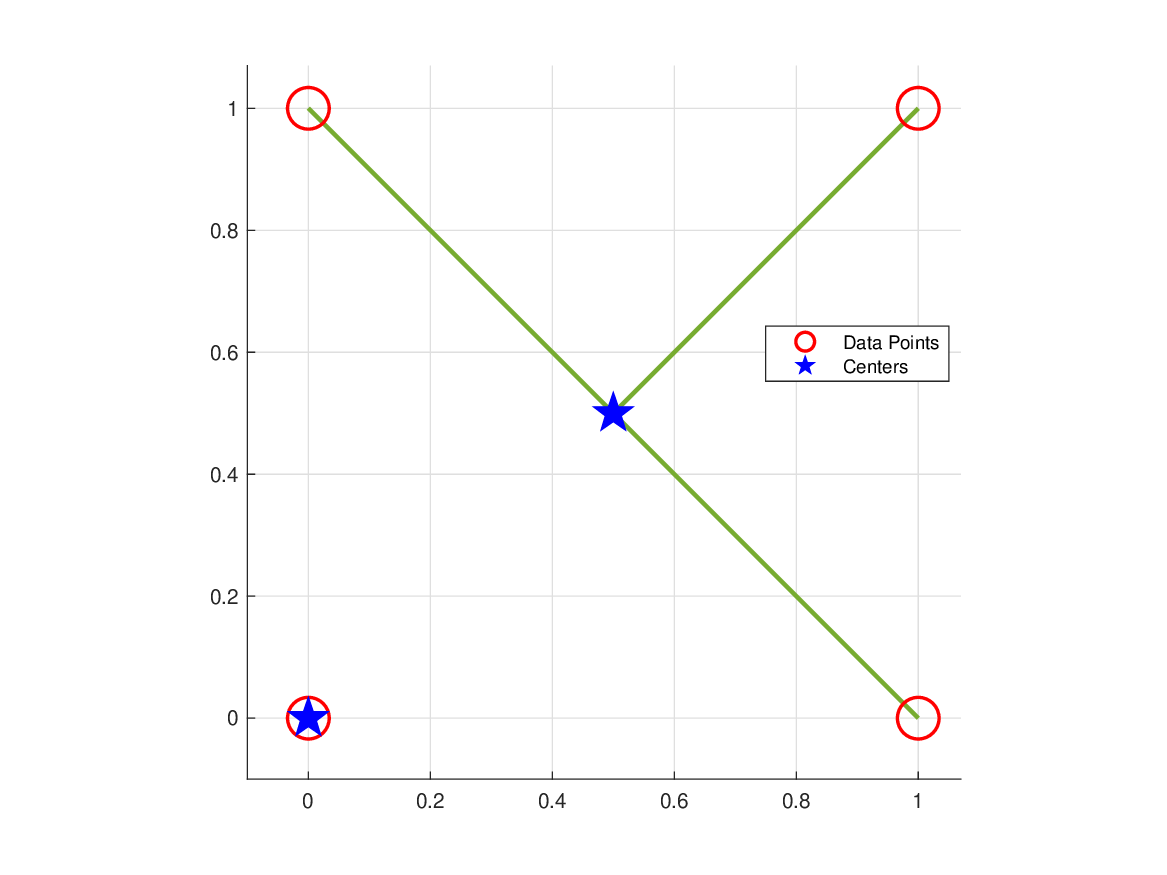}
            \caption{$\rho_F(x) = \|x\|_\infty$}
            \label{fig: 4points Linf}
        \end{subfigure}
        \caption{Visualization of the global solutions to \Cref{example: 4 points}.}
        \label{fig: 4 points}
    \end{figure}
\end{example}

\begin{example}\label{comparison tests} We revisit Example 3 from \cite{Nam2017}, which tests 6 real data sets: WINE (178 instances in $\R^{13}$, 3 different cultivars), IRIS (150 observations in $\R^4$, 3 iris types), PIMA (768 observations in $\R^8$, predicting diabetes), IONOSPHERE (351 radar observations in $\R^{34}$), and the latitude/longitude of 1217 US cities with 3 centers. A visualization of the US cities global solution with $\rho_F(x) = \|x\|_\infty$ is shown in \Cref{fig: USCity_Linf}. We compare Algorithm 5 from \cite{Nam2017} with \Cref{BDCA unconstrained,BDCA unconstrained - 2}, using the same $\mu$ values with suitable choices for $\lambda_{\text{skip}}$, $\lambda_{\text{start}} = 1$ and $\delta=0.5$. \Cref{table: 6 real datasets with skipping,table: 6 real datasets without skipping} present the average run-time and number of iterations over 100 runs, showing improvements across all norms. However, \Cref{table: 6 real datasets without skipping} shows that \Cref{BDCA unconstrained} has slower run-time despite fewer iterations. When using \Cref{BDCA unconstrained - 2}, see \Cref{table: 6 real datasets with skipping}, we have both run-time and iteration improvement. Further discussion is provided in the next example. 

 \begin{figure}[htb!]
	\centering
	\includegraphics[width=0.7\textwidth]{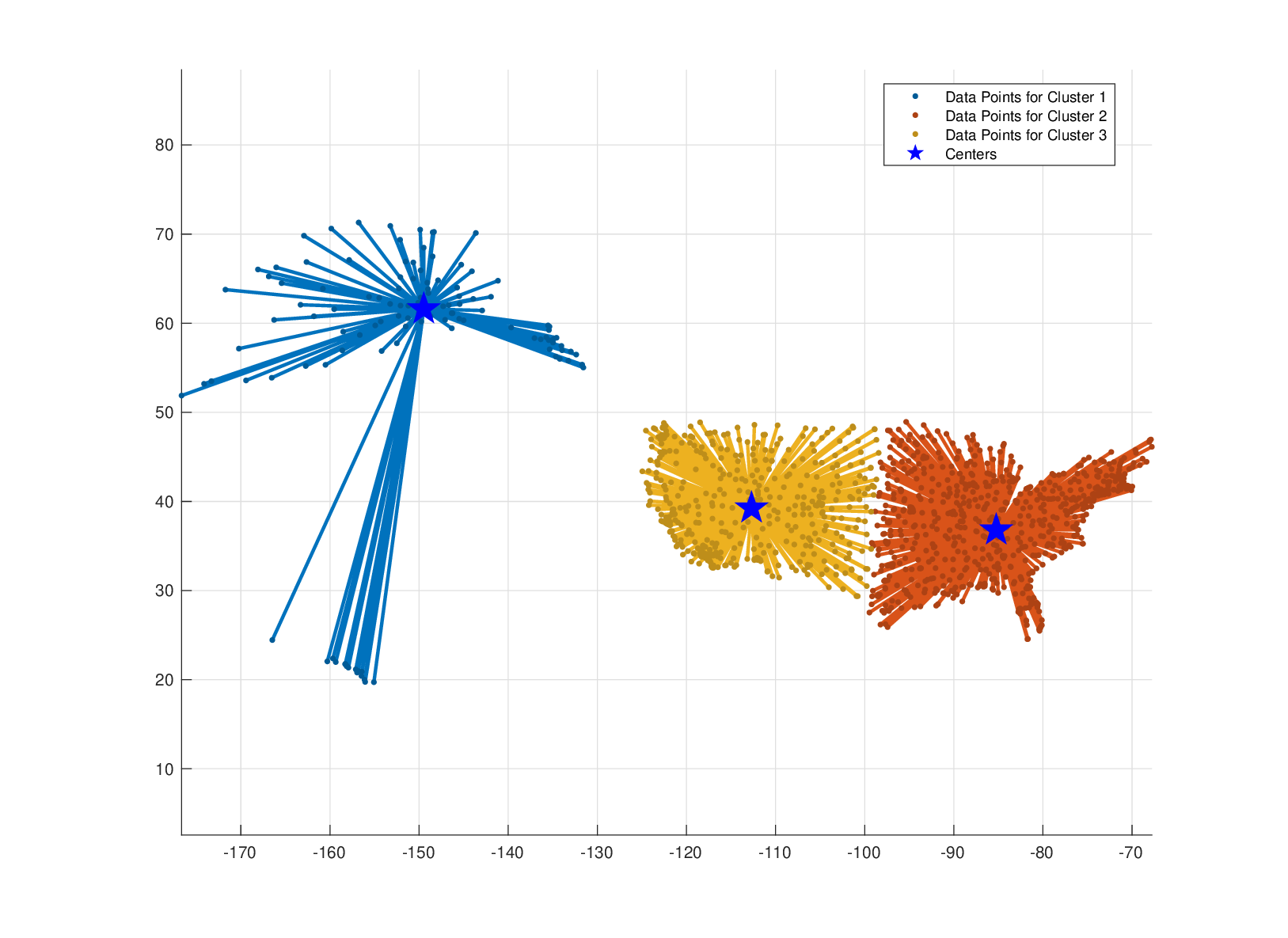}
	\caption{A visualization of the global solution to \Cref{comparison tests}.}
	\label{fig: USCity_Linf}
\end{figure}

\begin{table}[ht]
	\centering
  \footnotesize
  \resizebox{1\textwidth}{!}{%
	\begin{tblr}
		{
		column{6} = {c},
		column{7} = {c},
		column{8} = {c},
		column{9} = {c},
		column{10} = {c},
		cell{1}{5} = {c=2}{},
		cell{1}{7} = {c=2}{},
		cell{1}{9} = {c=2}{},
		cell{2}{2} = {c},
		cell{2}{3} = {c},
		cell{2}{4} = {c},
		cell{2}{5} = {c},
		cell{3}{2} = {c},
		cell{3}{3} = {c},
		cell{3}{4} = {c},
		cell{3}{5} = {c},
		cell{4}{2} = {c},
		cell{4}{3} = {c},
		cell{4}{4} = {c},
		cell{4}{5} = {c},
		cell{5}{2} = {c},
		cell{5}{3} = {c},
		cell{5}{4} = {c},
		cell{5}{5} = {c},
		cell{6}{2} = {c},
		cell{6}{3} = {c},
		cell{6}{4} = {c},
		cell{6}{5} = {c},
		cell{7}{2} = {c},
		cell{7}{3} = {c},
		cell{7}{4} = {c},
		cell{7}{5} = {c},
		vline{7,9,11} = {1}{},
		vline{5-6,8,10} = {1}{},
				vline{2-11} = {2}{},
				vline{-} = {3-7}{},
				hline{1} = {5-10}{},
				hline{2} = {2-10}{},
				hline{3-8} = {-}{},
			}
		           &            &    &   & $\rho_F(x) = \|x\|_2$ &      & $\rho_F(x) = \|x\|_1$ &      & $\rho_F(x) = \|x\|_\infty$ &      \\
		           & m          & n  & k & {Iteration                                                                                      \\
		Ratio}     & {Time                                                                                                                 \\
		Ratio}     & {Iteration                                                                                                            \\
		Ratio~}    & {Time                                                                                                                 \\
		Ratio~}    & {Iteration                                                                                                            \\
		Ratio~}    & {Time                                                                                                                 \\
		Ratio~}                                                                                                                            \\
		Wine       & 178        & 13 & 3 & 1.31                  & 0.65 & 1.54                  & 0.73 & 1.41                       & 0.58 \\
		Iris       & 150        & 4  & 3 & 1.69                  & 0.79 & 1.61                  & 0.61 & 1.23                       & 0.56 \\
		PIMA       & 768        & 8  & 2 & 1.59                  & 0.79 & 1.36                  & 0.65 & 2.18                       & 0.96 \\
		Ionosphere & 351        & 34 & 2 & 1.42                  & 0.82 & 1.05                  & 0.46 & 1.40                       & 0.65 \\
		USCity     & 1217       & 2  & 3 & 3.93                  & 1.54 & 1.80                  & 0.53 & 2.07                       & 0.79
	\end{tblr}
 }
\caption{DCA/aBDCA ratios for time and iterations in \Cref{comparison tests} using \cref{BDCA unconstrained}}
	\label{table: 6 real datasets without skipping}
\end{table}

\begin{table}[ht]
 \centering
 \footnotesize
 \resizebox{1\textwidth}{!}{%
	
	\begin{tblr}
		{
		column{6} = {c},
		column{7} = {c},
		column{8} = {c},
		column{9} = {c},
		column{10} = {c},
		cell{1}{5} = {c=2}{},
		cell{1}{7} = {c=2}{},
		cell{1}{9} = {c=2}{},
		cell{2}{2} = {c},
		cell{2}{3} = {c},
		cell{2}{4} = {c},
		cell{2}{5} = {c},
		cell{3}{2} = {c},
		cell{3}{3} = {c},
		cell{3}{4} = {c},
		cell{3}{5} = {c},
		cell{4}{2} = {c},
		cell{4}{3} = {c},
		cell{4}{4} = {c},
		cell{4}{5} = {c},
		cell{5}{2} = {c},
		cell{5}{3} = {c},
		cell{5}{4} = {c},
		cell{5}{5} = {c},
		cell{6}{2} = {c},
		cell{6}{3} = {c},
		cell{6}{4} = {c},
		cell{6}{5} = {c},
		cell{7}{2} = {c},
		cell{7}{3} = {c},
		cell{7}{4} = {c},
		cell{7}{5} = {c},
		vline{7,9,11} = {1}{},
		vline{5-6,8,10} = {1}{},
				vline{2-11} = {2}{},
				vline{-} = {3-7}{},
				hline{1} = {5-10}{},
				hline{2} = {2-10}{},
				hline{3-8} = {-}{},
			}
		           &            &    &   & $\rho_F(x) = \|x\|_2$ &      & $\rho_F(x) = \|x\|_1$ &      & $\rho_F(x) = \|x\|_\infty$ &      \\
		           & m          & n  & k & {Iteration                                                                                      \\
		Ratio}     & {Time                                                                                                                 \\
		Ratio}     & {Iteration                                                                                                            \\
		Ratio~}    & {Time                                                                                                                 \\
		Ratio~}    & {Iteration                                                                                                            \\
		Ratio~}    & {Time                                                                                                                 \\
		Ratio~}                                                                                                                            \\
		Wine       & 178        & 13 & 3 & 1.36                  & 1.27 & 1.51                  & 1.67 & 1.41                       & 1.27 \\
		Iris       & 150        & 4  & 3 & 1.75                  & 1.43 & 1.56                  & 1.20 & 1.31                       & 1.19 \\
		PIMA       & 768        & 8  & 2 & 2.02                  & 1.41 & 1.35                  & 1.13 & 2.09                       & 1.67 \\
		Ionosphere & 351        & 34 & 2 & 1.44                  & 1.32 & 1.09                  & 1.02 & 1.43                       & 1.33 \\
		USCity     & 1217       & 2  & 3 & 4.13                  & 3.71 & 1.77                  & 1.65 & 2.06                       & 1.91
	\end{tblr} }
\caption{DCA/aBDCA ratios for time and iterations in \Cref{comparison tests} using \cref{BDCA unconstrained - 2}}
	\label{table: 6 real datasets with skipping}
\end{table}
\end{example}

\begin{example}\label{4centers}
	We consider an artificial example similar to Example 7.3 from \cite{Nam2018}, involving 1000 points within 4 circles, where 4 centers are found constrained to a single center circle (see in \Cref{fig:1000pts}). The parameters are $\sigma=10,\delta=0.25,\lambda_{\text{start}}=0.1$, \(\lambda_{f}=10^{-3}\), \(\lambda_{\text{skip}}=30\). Figures \ref{4centers_compare_withskip} and  \ref{4centers_compare_w/oskip} compare adaptive BDCA and DCA using \Cref{BDCA constrained,BDCA constrained - 2} over 100 random initial values. The need for skipping becomes clear, as both show similar iteration improvements. However, performing \cref{BDCA constrained} makes it 2/3 slower than DCA, while  \Cref{BDCA constrained - 2} offers a run-time advantage. This can be explained by examining the line search step lengths in both BDCA versions.
 
 \begin{figure}[htb!]
	\centering
	\includegraphics[width=0.6\textwidth]{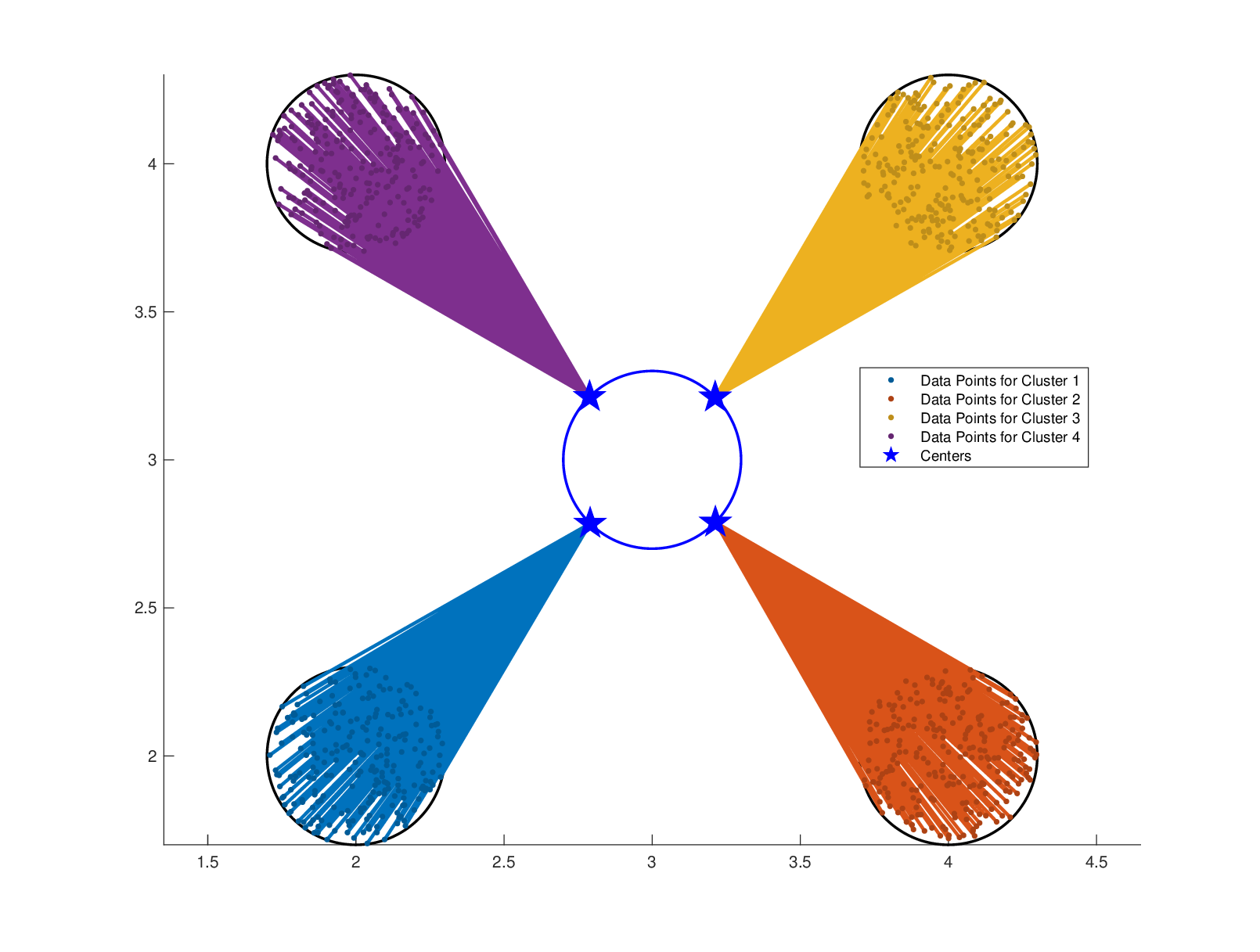}
	\caption{Visualization of the global solution to \Cref{4centers}.}
	\label{fig:1000pts}
\end{figure}

\begin{figure}[htb!]
	\centering
	\begin{subfigure}[b]{0.42\textwidth}
		\centering
		\includegraphics[width=\textwidth]{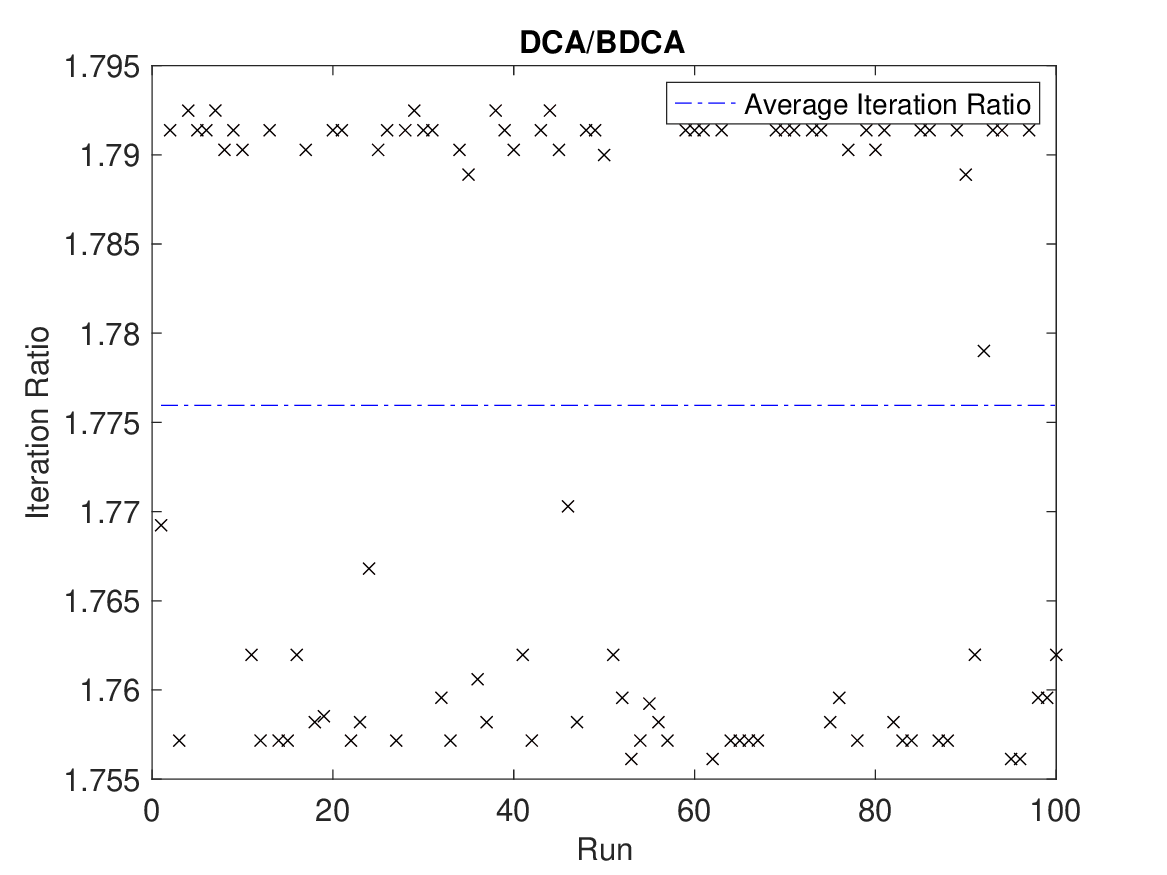}
	\end{subfigure}
	%\hfill
	\begin{subfigure}[b]{0.42\textwidth}
		\centering
		\includegraphics[width=\textwidth]{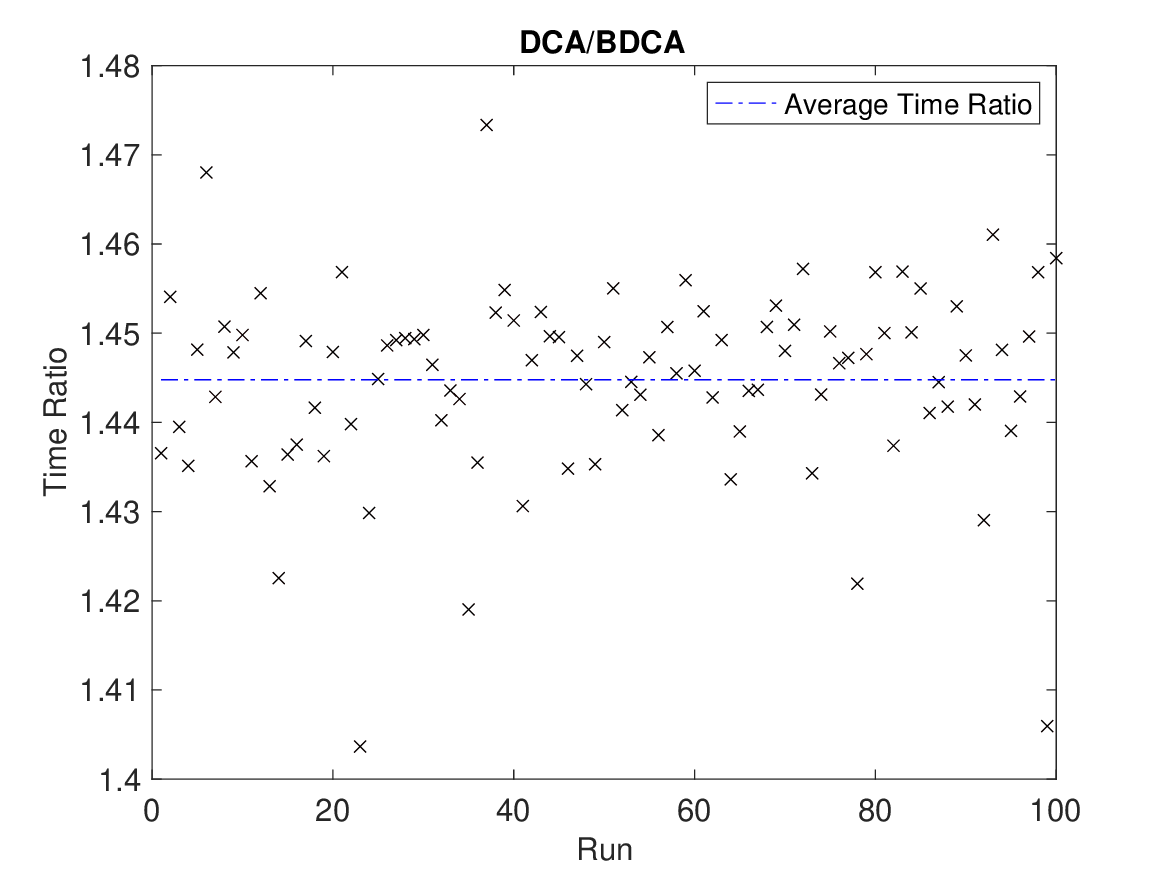}
	\end{subfigure}
	\caption{Iteration and time ratio comparison for \Cref{4centers} with \cref{BDCA constrained - 2}.}
	\label{4centers_compare_withskip}
	\vskip\baselineskip
	\begin{subfigure}[b]{0.42\textwidth}
		\centering
		\includegraphics[width=\textwidth]{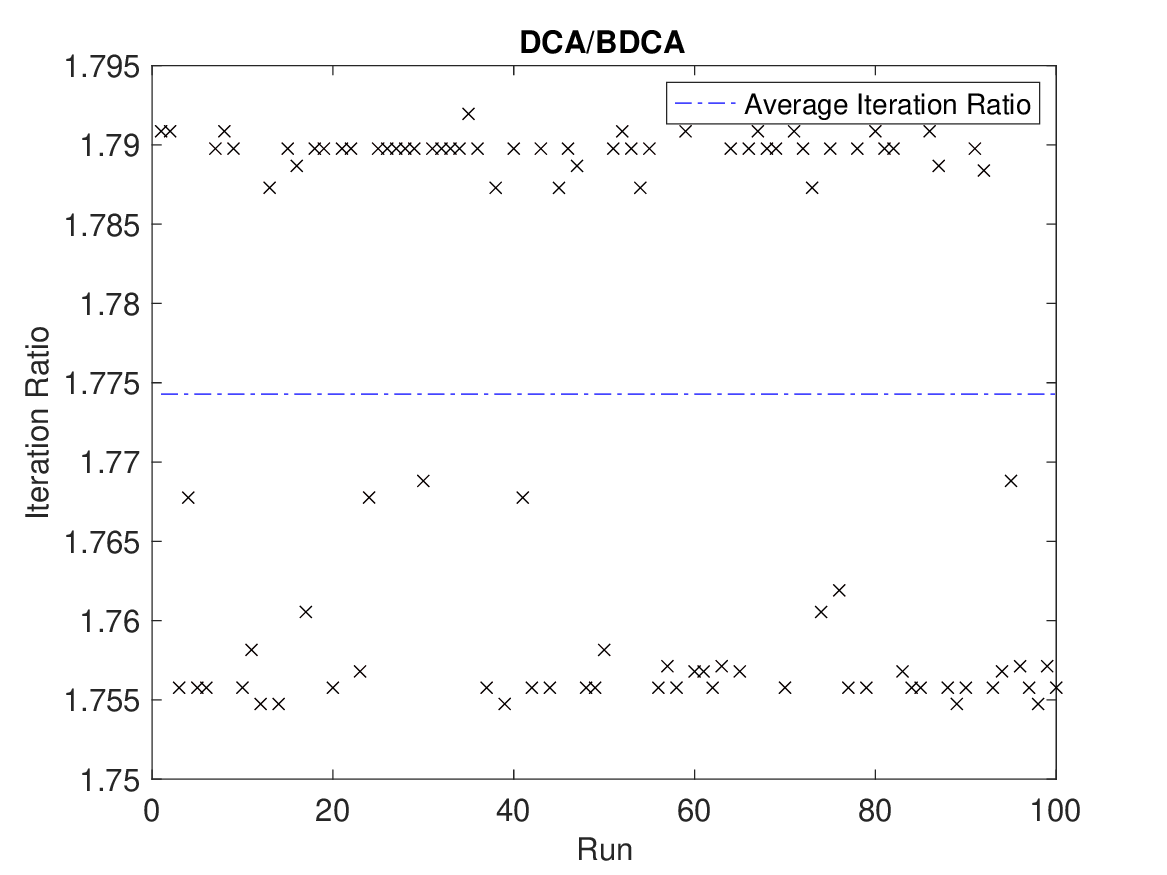}
	\end{subfigure}
	%\hfill
	\begin{subfigure}[b]{0.42\textwidth}
		\centering
		\includegraphics[width=\textwidth]{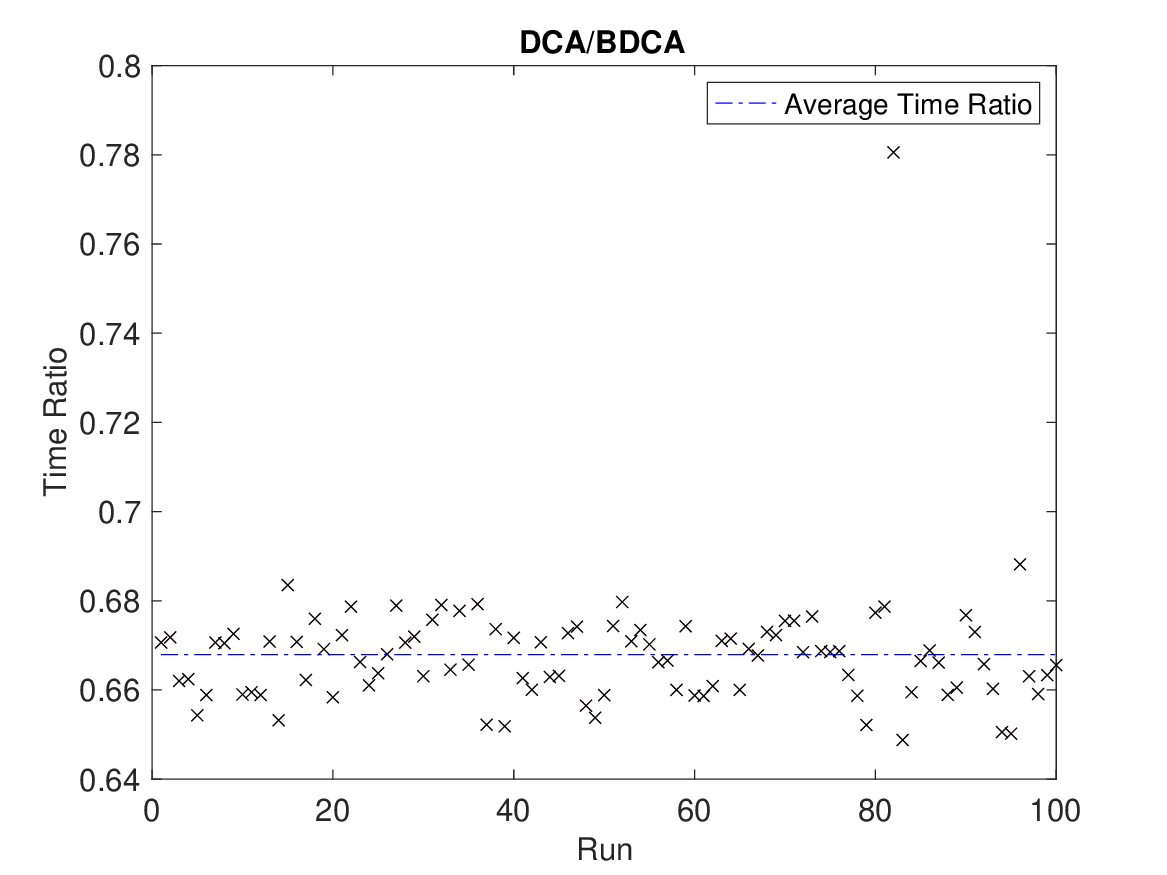}
	\end{subfigure}
	\caption{Iteration and time ratio comparison for \Cref{4centers} with \cref{BDCA constrained}}
	\label{4centers_compare_w/oskip}
\end{figure}

\begin{figure}[htb!]
	\centering
	\begin{subfigure}[b]{0.42\textwidth}
		\centering
		\includegraphics[width=\textwidth]{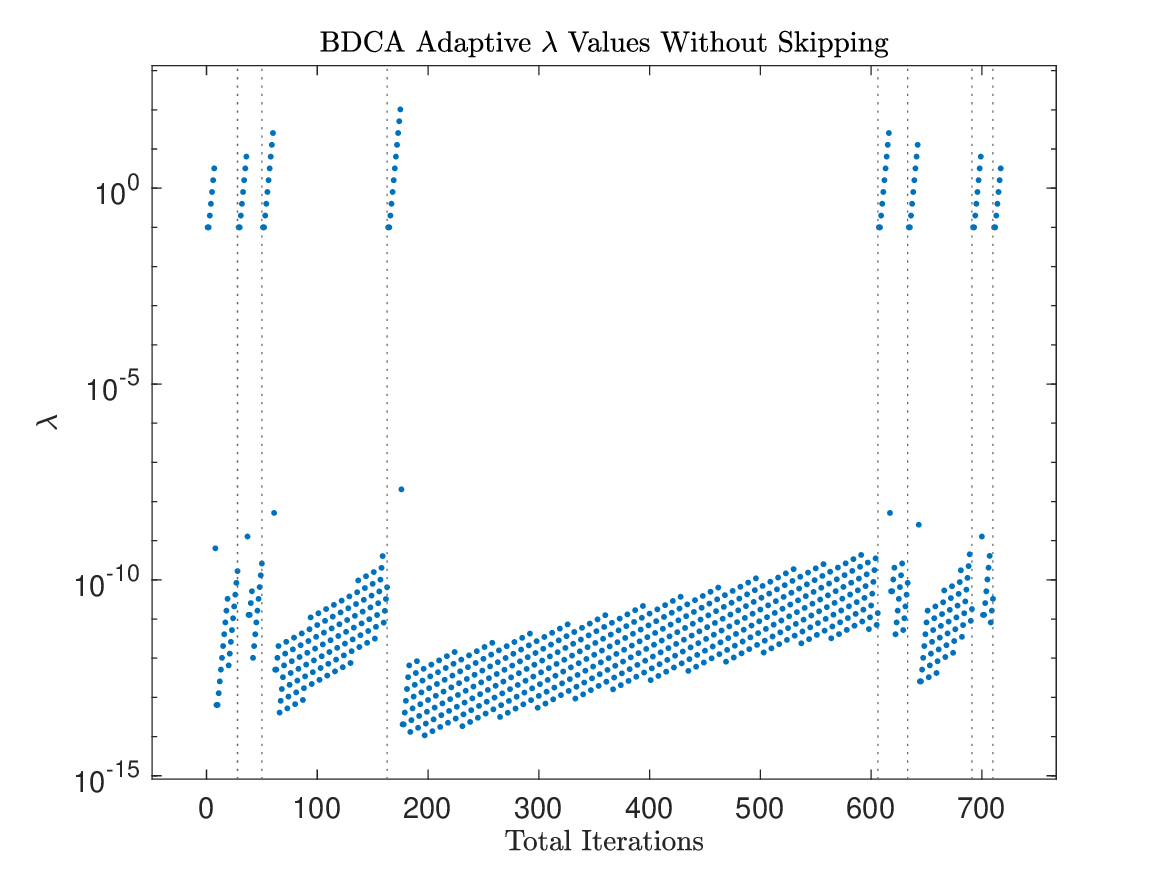}
	\end{subfigure}
	%\hfill
	\begin{subfigure}[b]{0.42\textwidth}
		\centering
		\includegraphics[width=\textwidth]{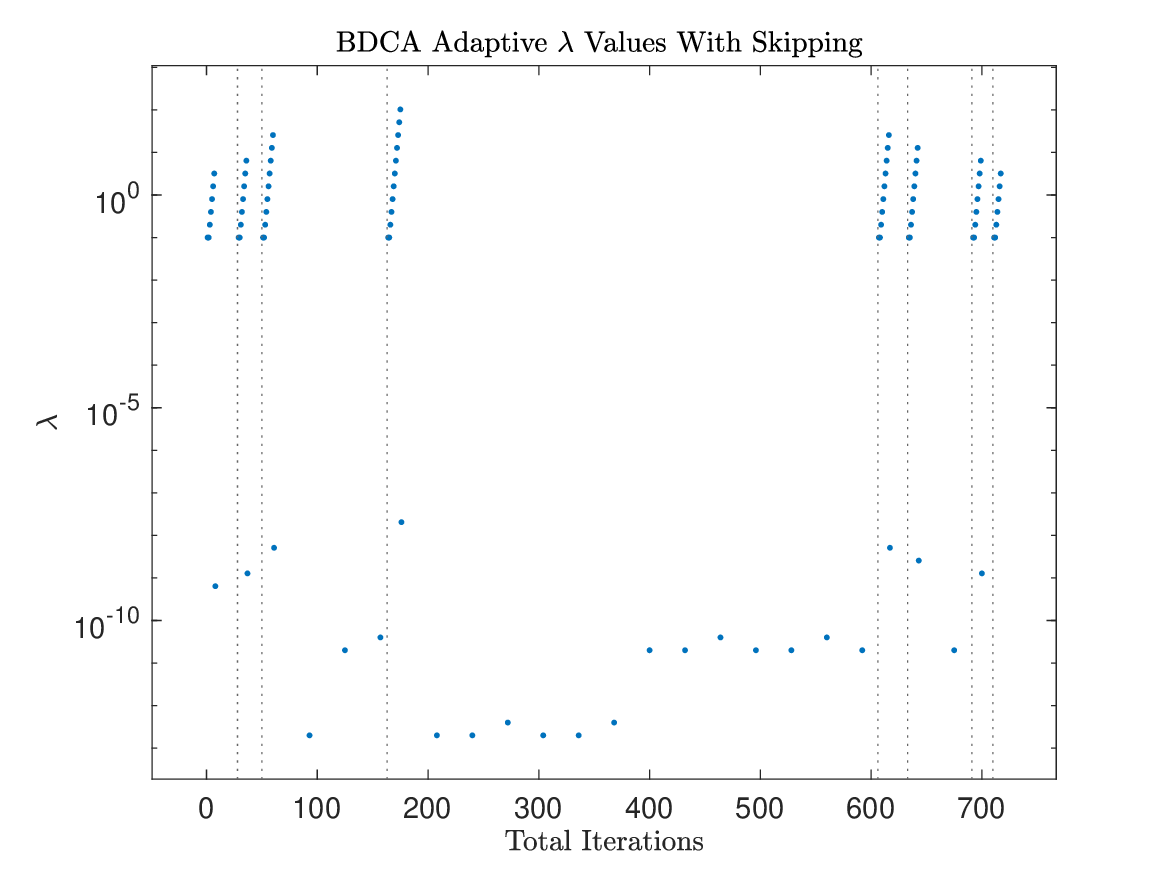}
	\end{subfigure}
	\caption{Line-search step lengths \(\lambda\) for a single run of \Cref{4centers} using \cref{BDCA constrained} (left), and \cref{BDCA constrained - 2} (right). Dashed vertical lines indicate outer loop iteration, adjusting \(\tau\) and \(\mu\).}
	\label{fig: 7_3 lambda skip plot}
\end{figure}

\begin{figure}[htb!]
	\centering
	\begin{subfigure}[b]{0.42\textwidth}
		\centering
		\includegraphics[width=\textwidth]{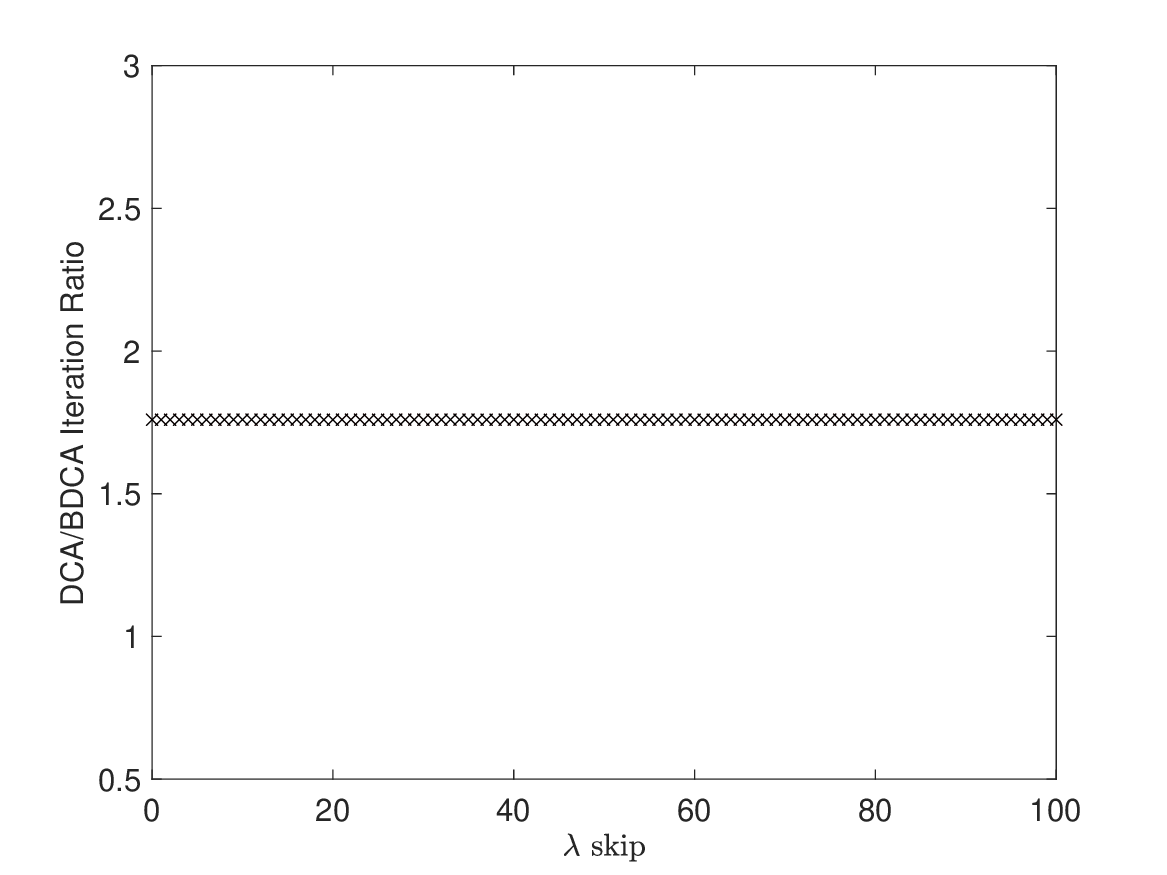}
	\end{subfigure}
	%\hfill
	\begin{subfigure}[b]{0.42\textwidth}
		\centering
		\includegraphics[width=\textwidth]{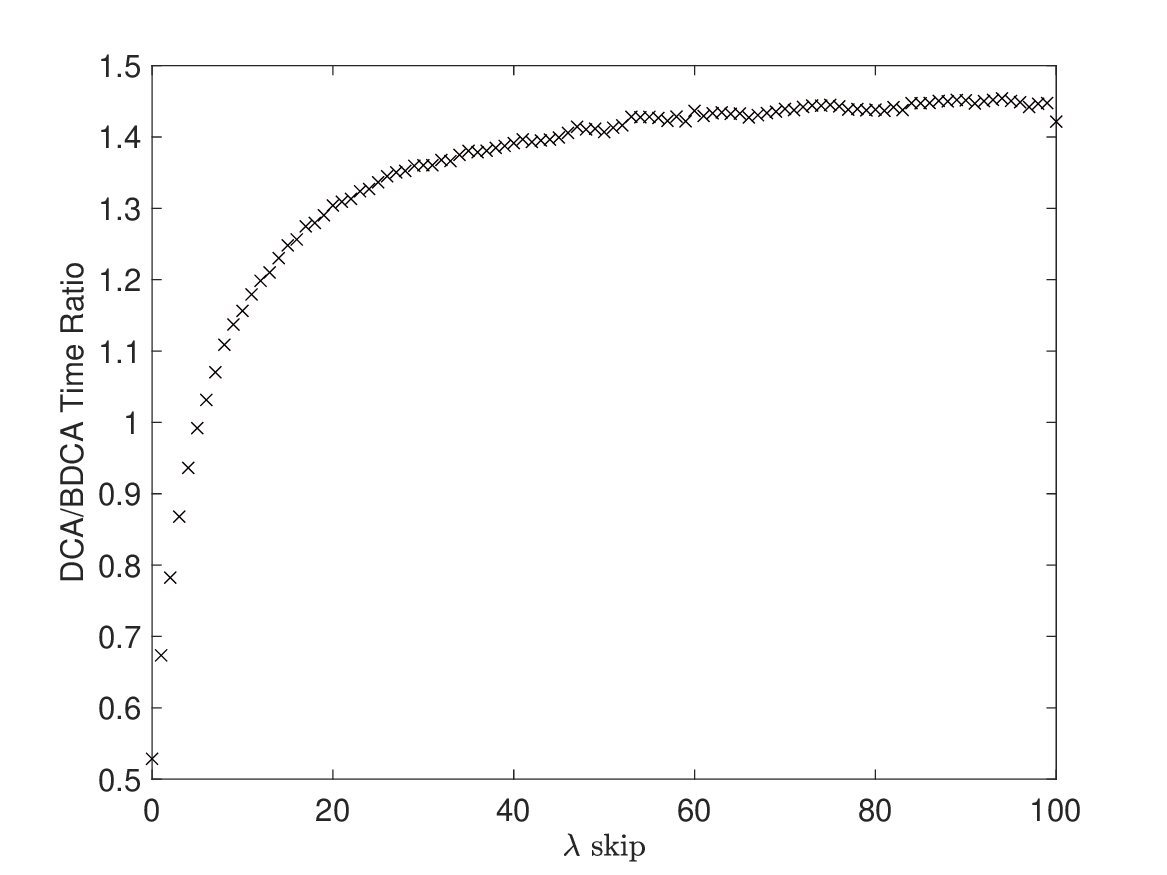}
	\end{subfigure}
	\caption{Average DCA/ aBDCA iteration and time ratio for varying values of $\lambda_{\text{skip}}$ in \cref{BDCA constrained - 2} for \Cref{4centers}.}
	\label{fig: 7_3 skipping plot}
\end{figure}

As shown in \Cref{fig: 7_3 lambda skip plot}, the step length \(\lambda\) is only significant right after adjusting $\tau$ and $\mu$ and becomes negligible after a few iterations, making adaptive BDCA behave like basic DCA. By skipping the line search every 30 iterations when $\lambda$ falls below \(10^{-3}\), we retain the benefits of line search without excessive cost. This pattern of narrow regions with non-trivial line search was consistent across all tested examples, proving the effectiveness of \Cref{BDCA constrained - 2} for MWP. However, such regions don’t always appear immediately after adjusting \(\tau \text{ and } \mu\), as shown in \Cref{example: US} and \Cref{fig: 7_4 lambda skip plot}. This is why we need a skipping algorithm rather than switching to basic DCA when $\lambda$ falls below $\lambda_f$. Finally, \Cref{fig: 7_3 skipping plot} shows that large values of \(\lambda_{\text{skip}}\) in \Cref{BDCA constrained - 2} have no negative impact, and skipping leads to significant performance gains over DCA in this case.
\end{example}

\begin{example}\label{example: US}
	We revisit Example 7.4 from \cite{Nam2018}, which uses real-world data consisting of the latitudes and longitudes of the 988 most-populated cities in the contiguous United States \cite{usmaps}. A visualization of the problem and its solution is provided in \cite{Nam2018}. The parameters are \(\sigma = 100, \; \delta =0.25, \; \lambda_{\text{start}}=1, \; \lambda_{f} = 10^{-3}, \; \lambda_{\text{skip}}=100\). \Cref{fig: US compare with skip,fig: US compare without skip} compare the performance of \cref{BDCA constrained,BDCA constrained - 2}, using 100 random initial values. Skipping is again crucial: while both algorithms reduce the iteration count by approximately 1.7 times compared to DCA, \cref{BDCA constrained} is 0.8 times slower due to the line search cost. With \cref{BDCA constrained - 2}, however, we achieve a 1.6 run-time improvement over DCA. \Cref{fig: 7_4 lambda skip plot} illustrates the line search step lengths for both versions. Again, distinct narrow regions emerge where line search improves adaptive BDCA, while for most iterations, adaptive BDCA behaves like DCA. Unlike previous examples, these regions occur throughout an inner loop with fixed $\tau$ and $\mu$ values. The skipping version effectively detects and activates the line search when necessary, maintaining similar iteration improvements to \cref{BDCA constrained}. Figure \ref{fig: 7_4 skipping plot} shows the impact of varying $\lambda_{\text{skip}}$. The more complex nature of this problem is reflected in the noticeable sensitivity to $\lambda_{\text{skip}}$, but across a wide range of values, most yield a run-time improvement of around 1.6, as long as $\lambda_{\text{skip}}$ exceeds 20. This highlights the robustness of the algorithm to the choice of $\lambda_{\text{skip}}$. 
\end{example}

\begin{figure}[htb!]
	\centering
	\begin{subfigure}[b]{0.42\textwidth}
		\centering
		\includegraphics[width=\textwidth]{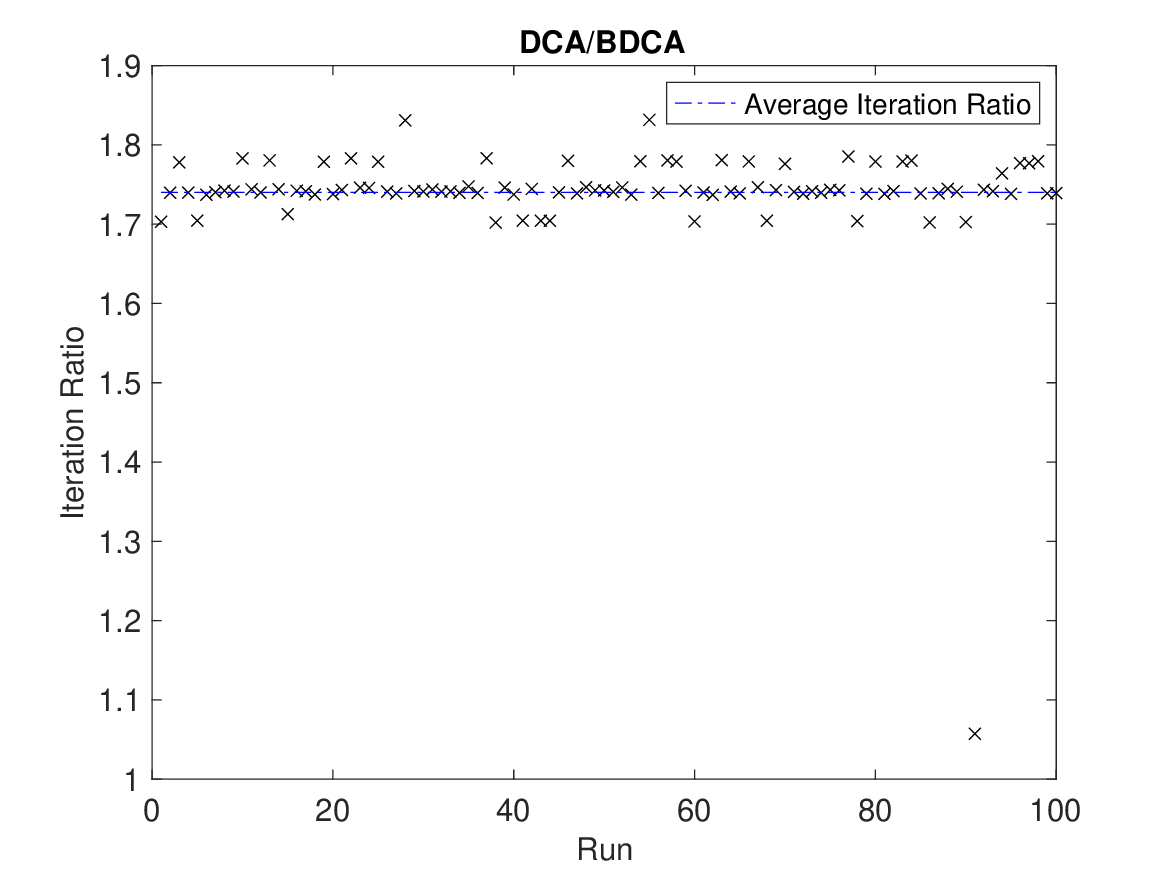}
	\end{subfigure}
	%\hfill
	\begin{subfigure}[b]{0.42\textwidth}
		\centering
		\includegraphics[width=\textwidth]{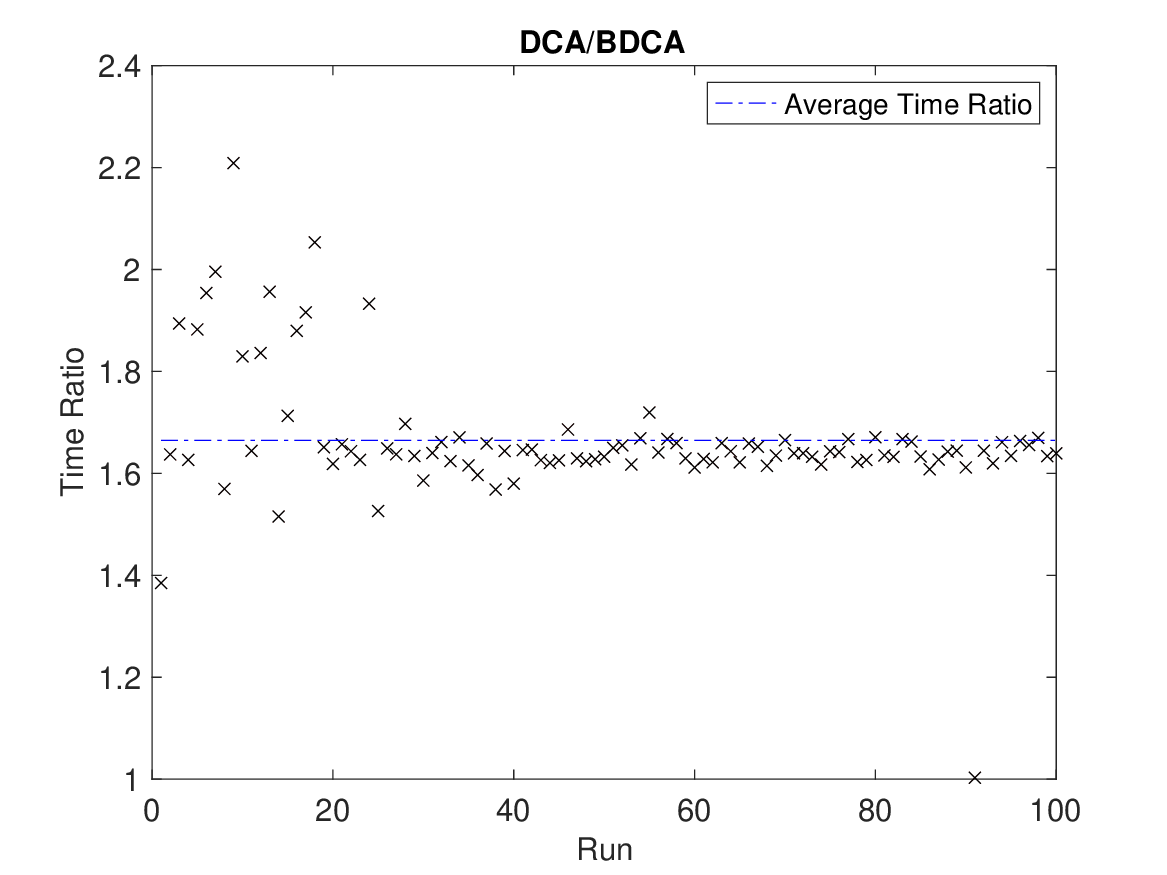}
	\end{subfigure}
	\caption{Iteration and time ratio comparison for \Cref{example: US} using \cref{BDCA constrained - 2}.}
	\label{fig: US compare with skip}
	\vskip\baselineskip
	\begin{subfigure}[b]{0.42\textwidth}
		\centering
		\includegraphics[width=\textwidth]{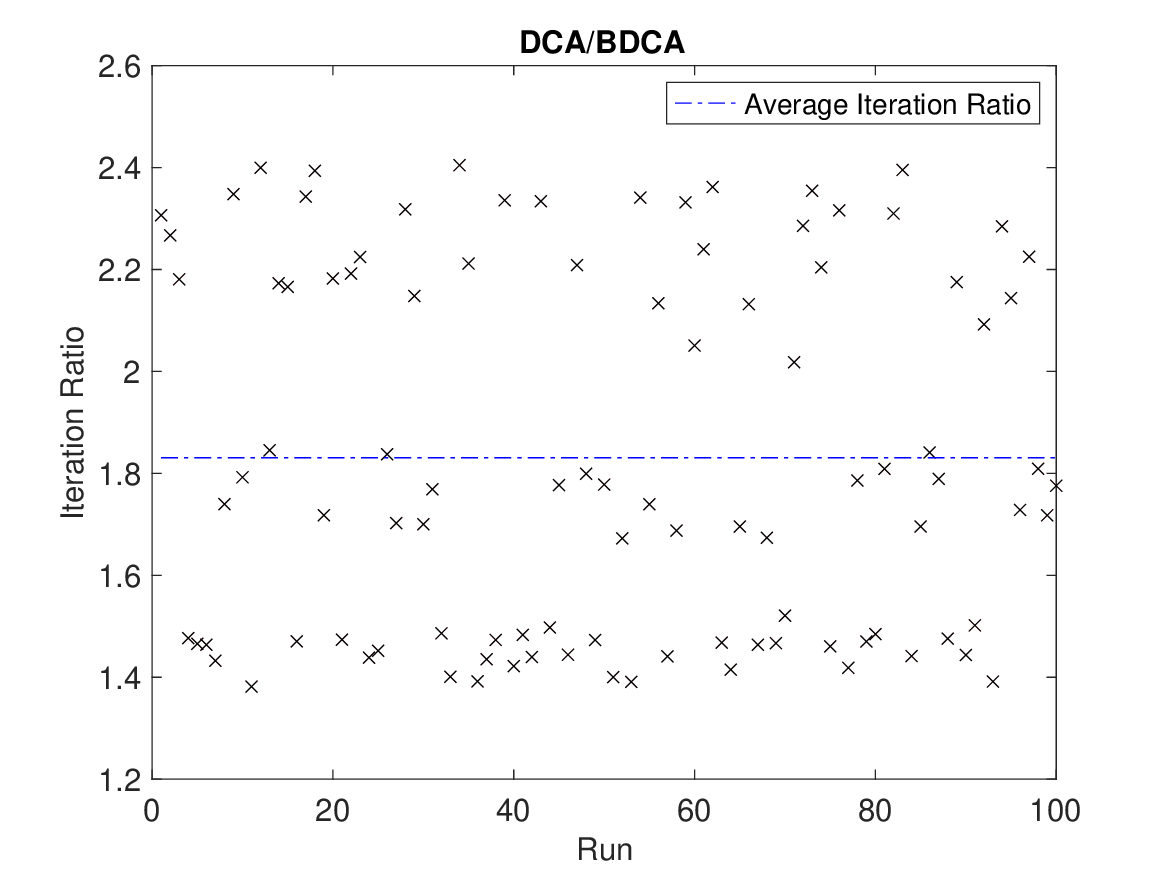}
	\end{subfigure}
	%\hfill
	\begin{subfigure}[b]{0.42\textwidth}
		\centering
		\includegraphics[width=\textwidth]{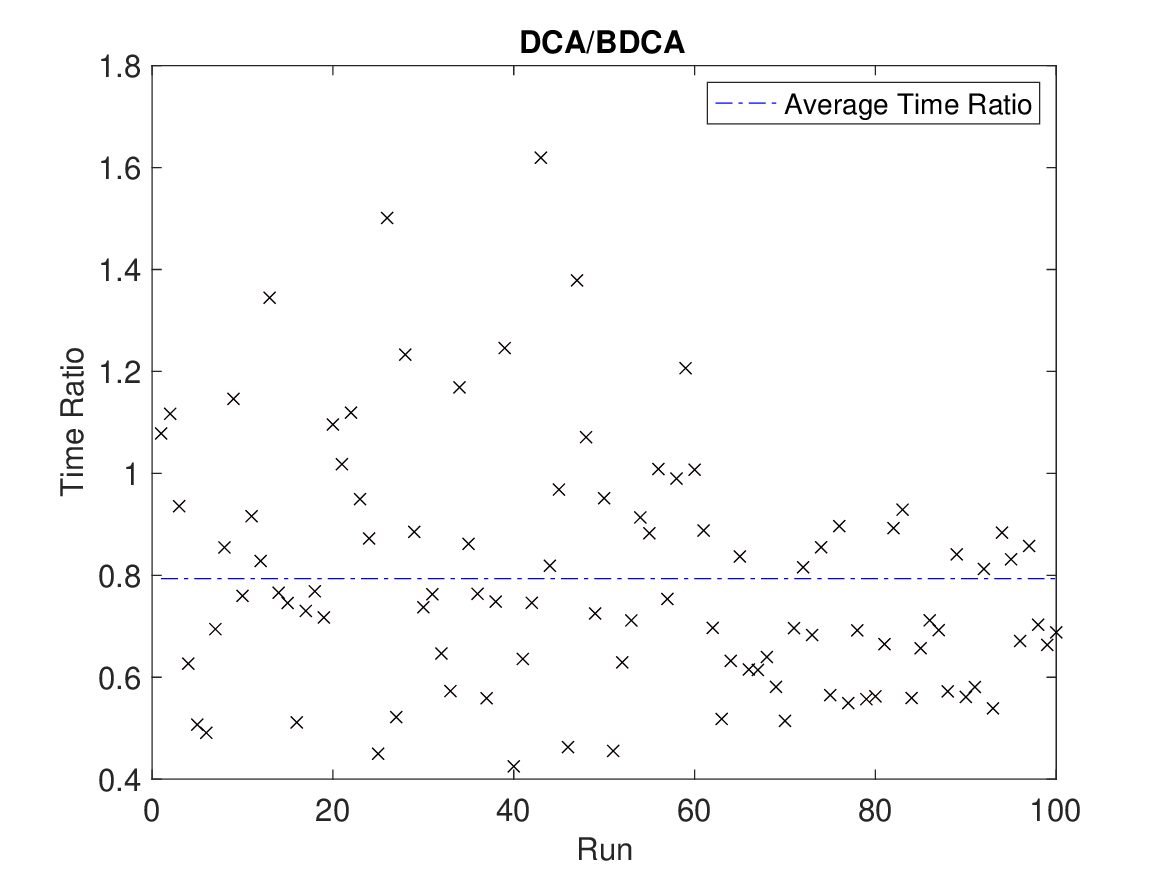}
	\end{subfigure}
	\caption{Iteration and time ratio comparison for \Cref{example: US} using \cref{BDCA constrained}.}
	\label{fig: US compare without skip}
\end{figure}

\begin{figure}[htb!]
	\centering
	\begin{subfigure}[b]{0.42\textwidth}
		\centering
		\includegraphics[width=\textwidth]{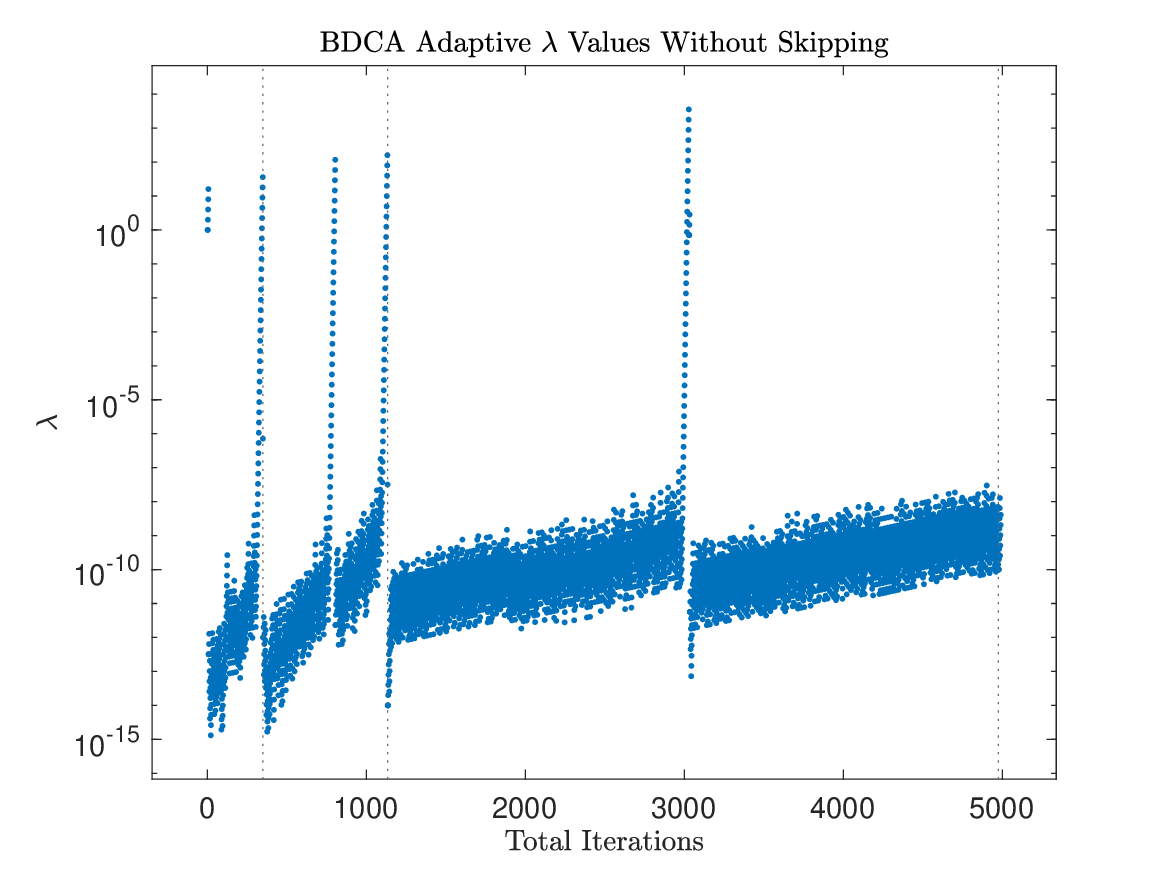}
	\end{subfigure}
	%\hfill
	\begin{subfigure}[b]{0.42\textwidth}
		\centering
		\includegraphics[width=\textwidth]{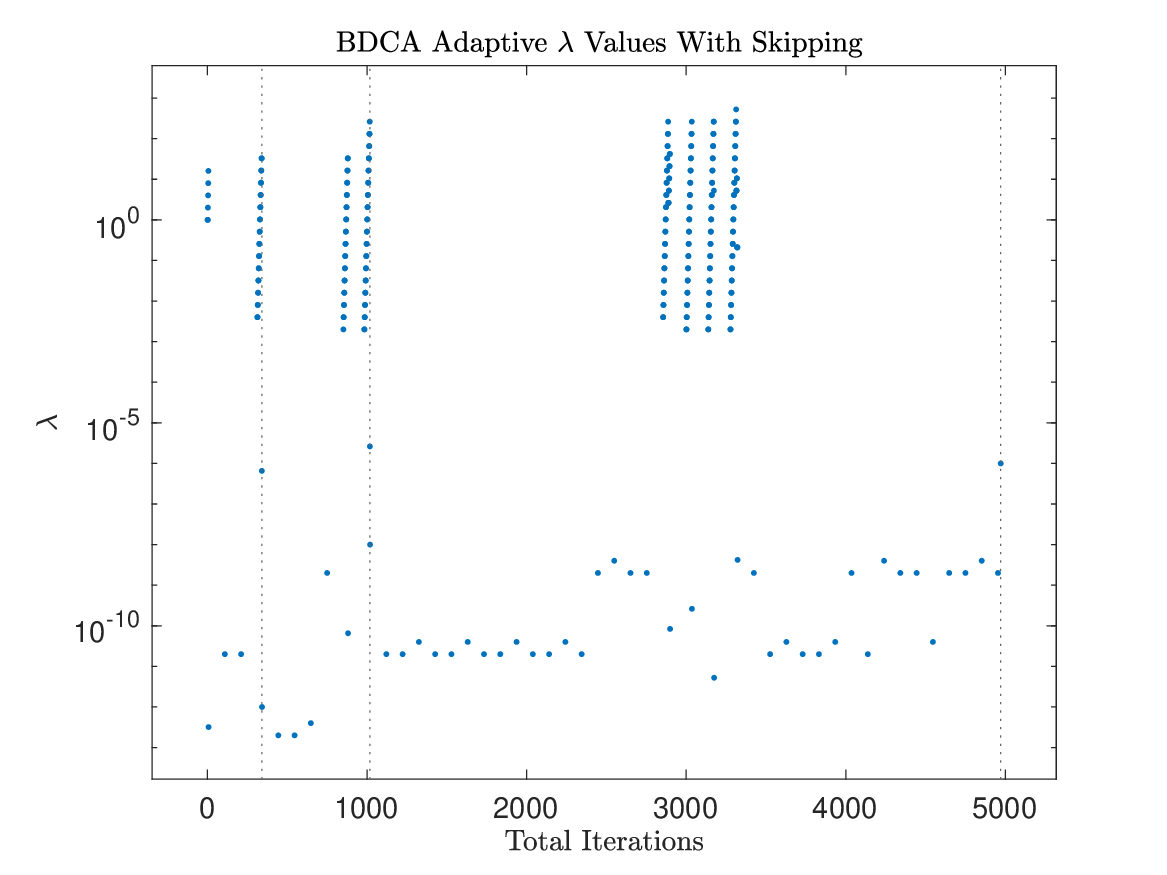}
	\end{subfigure}
	\caption{Line-search step lengths \(\lambda\) for a single run of \Cref{example: US} using \cref{BDCA constrained} (left) and \cref{BDCA constrained - 2} (right). Dashed vertical lines indicate outer loop iteration, adjusting \(\tau\) and \(\mu\).}
	\label{fig: 7_4 lambda skip plot}
\end{figure}

\begin{figure}[htb!]
	\centering
	\begin{subfigure}[b]{0.42\textwidth}
		\centering
		\includegraphics[width=\textwidth]{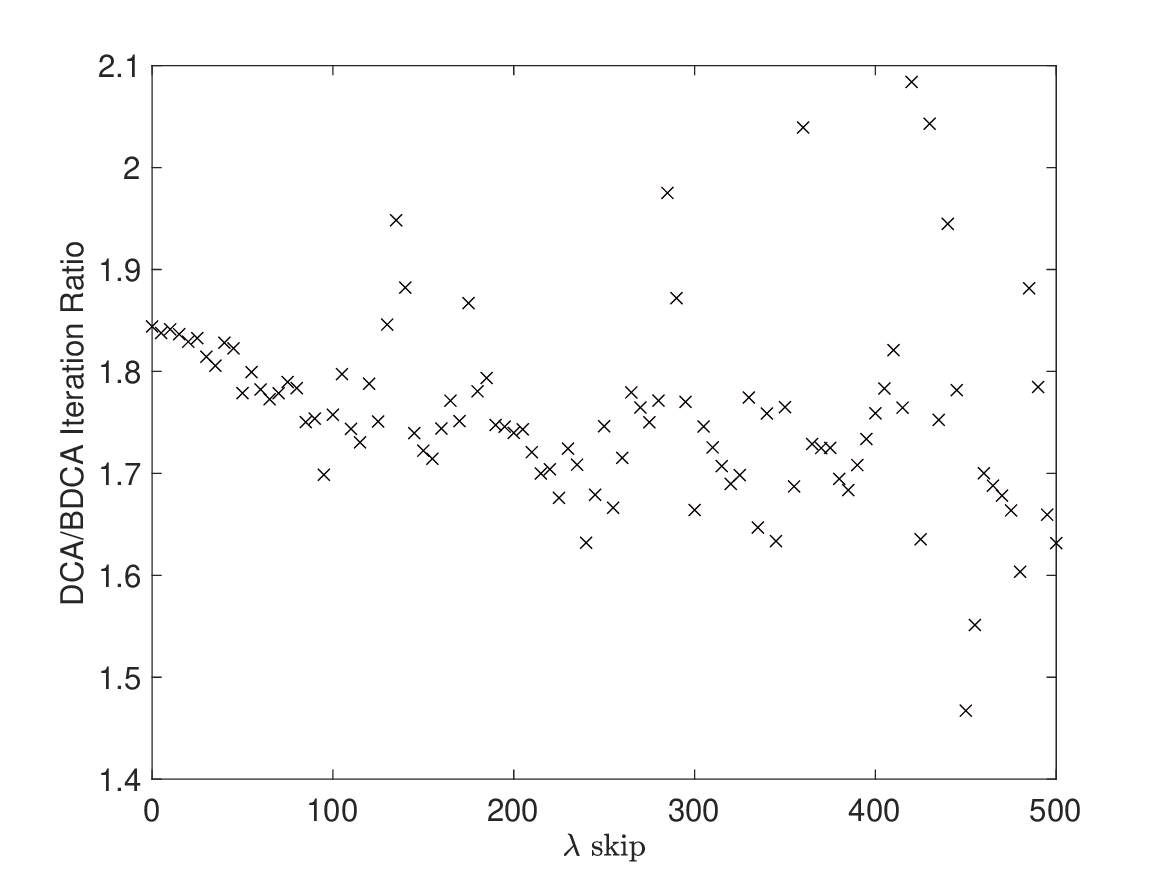}
	\end{subfigure}
	%\hfill
	\begin{subfigure}[b]{0.42\textwidth}
		\centering
		\includegraphics[width=\textwidth]{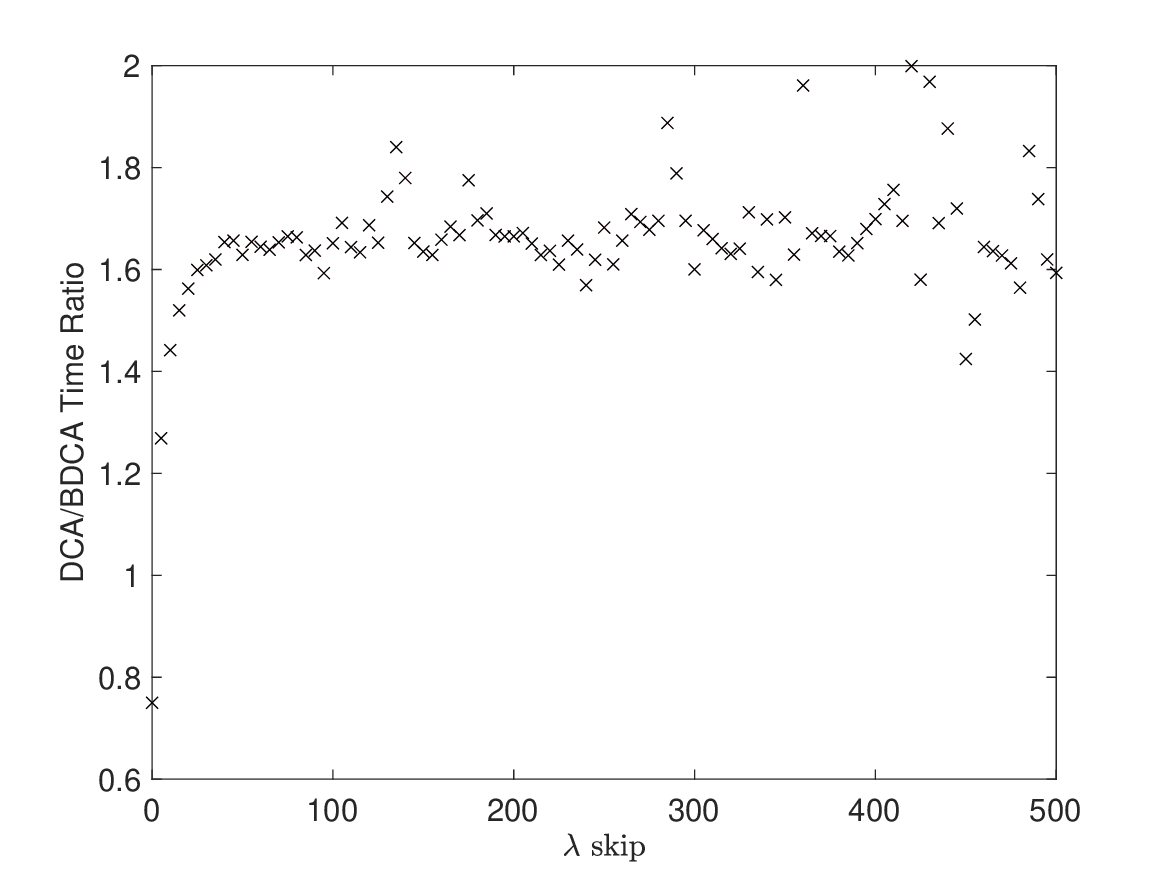}
	\end{subfigure}
	\caption{Average DCA/aBDCA iteration and time ratio for \Cref{example: US} for varying values of $\lambda_{\text{skip}}$ in \cref{BDCA constrained - 2}.}
	\label{fig: 7_4 skipping plot}
\end{figure}

\begin{example}\label{example: traveling salesman}
	We test our algorithm using the EIL76 dataset from the Traveling Salesman Problem Library \cite{Reinelt1991}, seeking 2 centers constrained to lie within:
	\begin{enumerate}
		\item Circles with radii 7 and 4.5, centered at \(\left( 30,40 \right) \text{ and } \left( 32,37\right) \), and the rectangle with vertices \(\{\left( 40,30 \right) ,\left( 40,40 \right), \left( 30,40 \right) ,\left( 30,30 \right) \} \) \\[-0.2in]
		\item Circles with radii 7,7, and 5, centered at \(\left( 35,20 \right) \), \(\left( 45,20 \right) \) and \(\left( 40,15 \right) \).
	\end{enumerate}
	A visualization of the problem and its solution is shown in \Cref{fig: traveling salesman}. The parameters are \(\sigma=100,\; \delta=0.85, \; \lambda_{\text{start}}=1, \; \lambda_{f} = 10^{-3}, \; \lambda_{\text{skip}}=150\). Only the results using \Cref{BDCA constrained - 2} are presented, as \Cref{BDCA constrained}, similar to previous examples, fails to provide a run-time improvement. Figure \ref{fig: traveling salesman compare with skip} shows the relative performance of \Cref{BDCA constrained - 2} compared to DCA, with a moderate run-time improvement of about 1.34 times. In addition, \Cref{fig: 76 cities skipping plot} shows that a wide range of \(\lambda_{\text{skip}}\)  values still result in a speed-up over DCA. 
\end{example}

 \begin{figure}[htb!]
	\centering
	\includegraphics[width=0.65\textwidth]{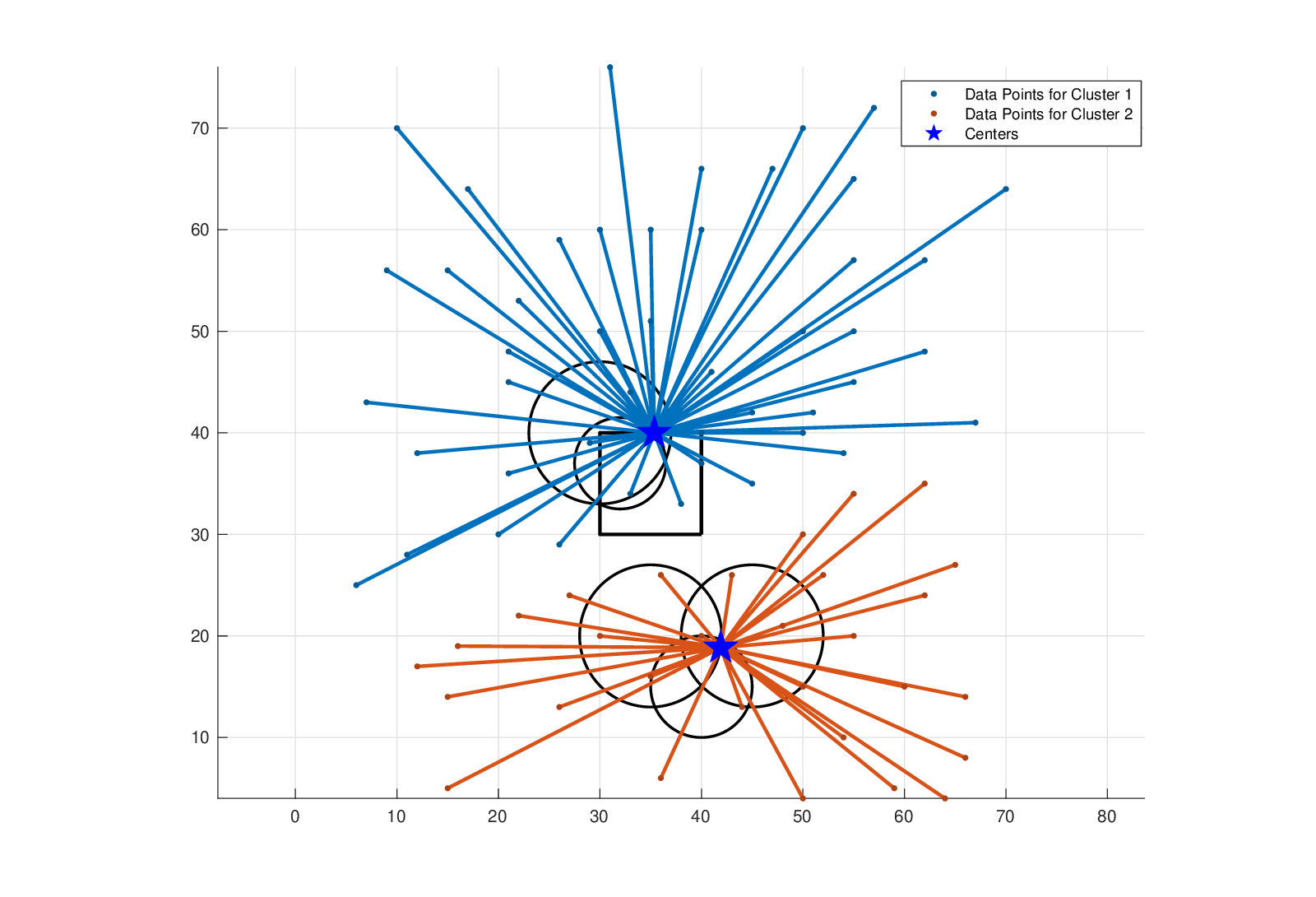}
	\caption{Visualization of the global solution to \Cref{example: traveling salesman}.}
	\label{fig: traveling salesman}
\end{figure}

\begin{figure}[htpb!]
	\centering
	\begin{subfigure}[b]{0.42\textwidth}
		\centering
		\includegraphics[width=\textwidth]{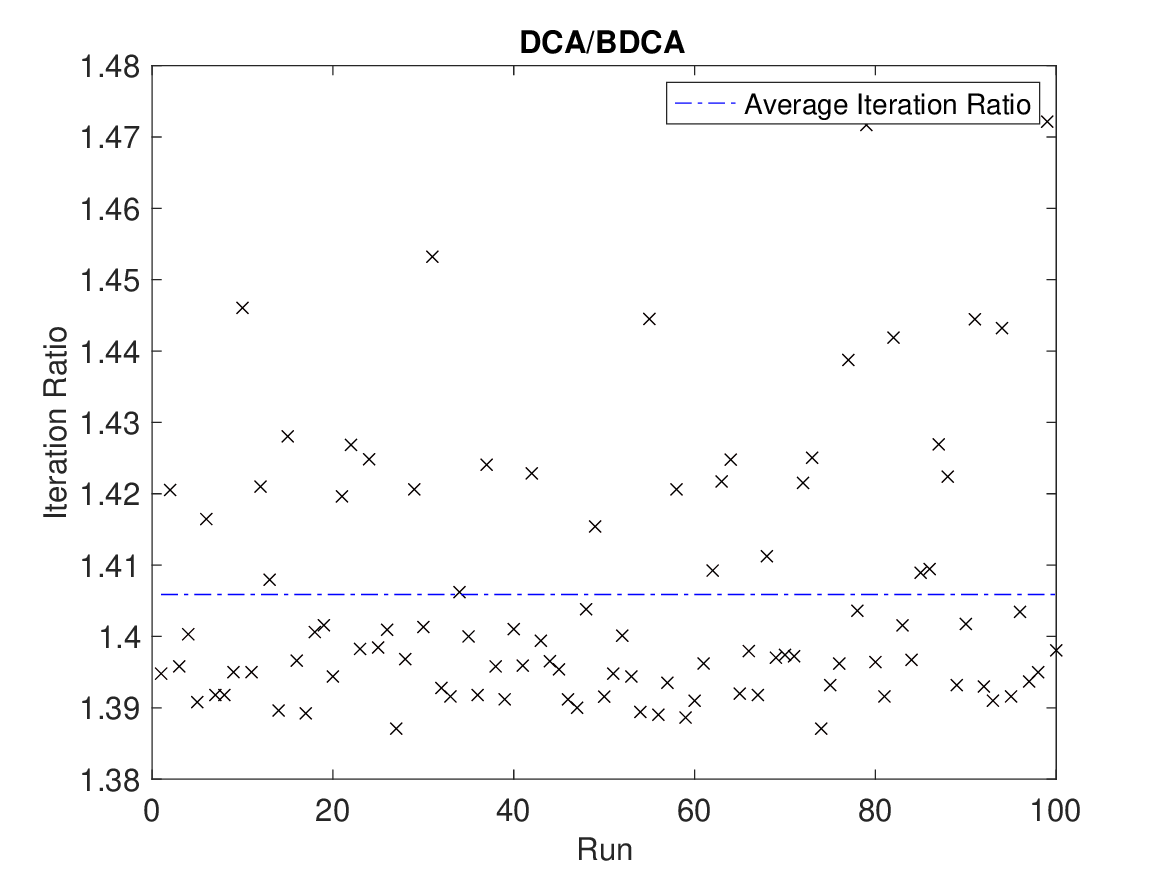}
	\end{subfigure}
	%\hfill
	\begin{subfigure}[b]{0.42\textwidth}
		\centering
		\includegraphics[width=\textwidth]{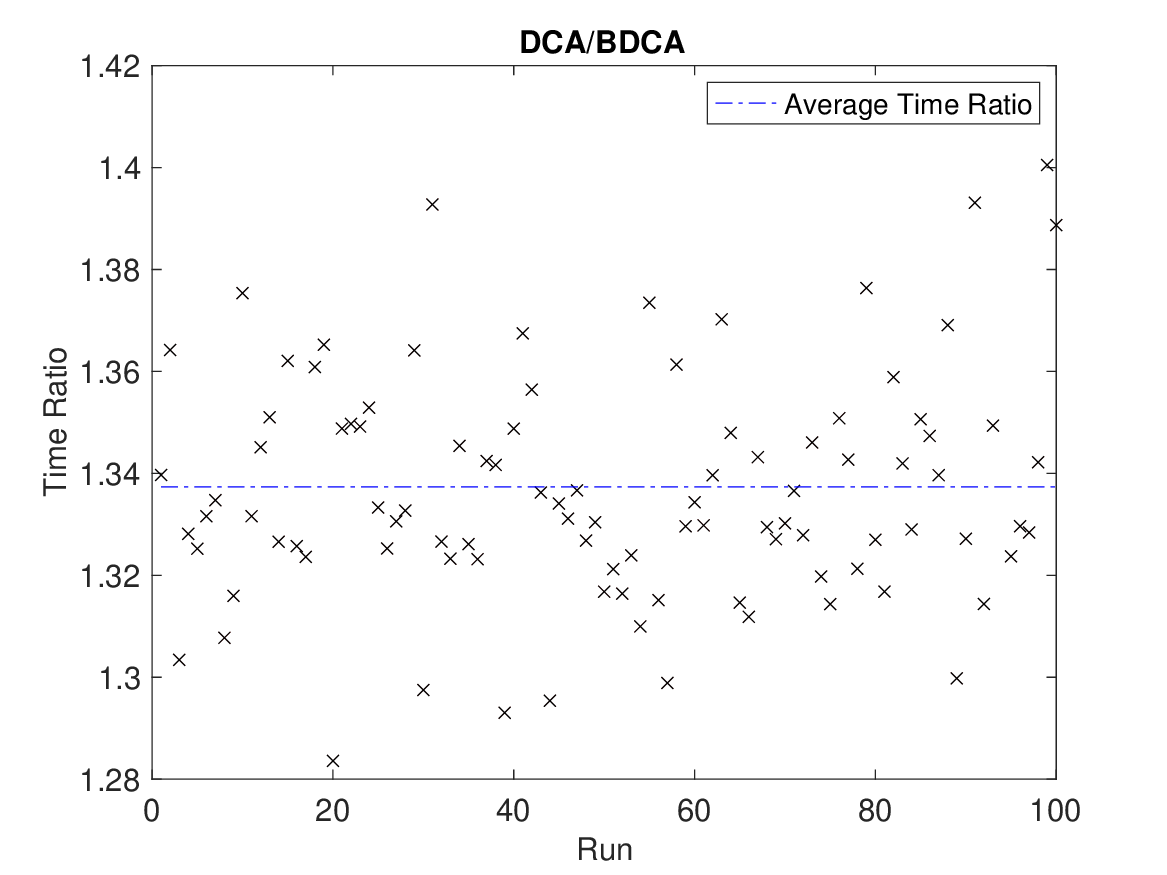}
	\end{subfigure}
	\caption{Iteration and time ratio comparison for \Cref{example: traveling salesman} using \cref{BDCA constrained - 2}.}
	\label{fig: traveling salesman compare with skip}
\end{figure}

\begin{figure}[htpb!]
	\centering
	\begin{subfigure}[b]{0.42\textwidth}
		\centering
		\includegraphics[width=\textwidth]{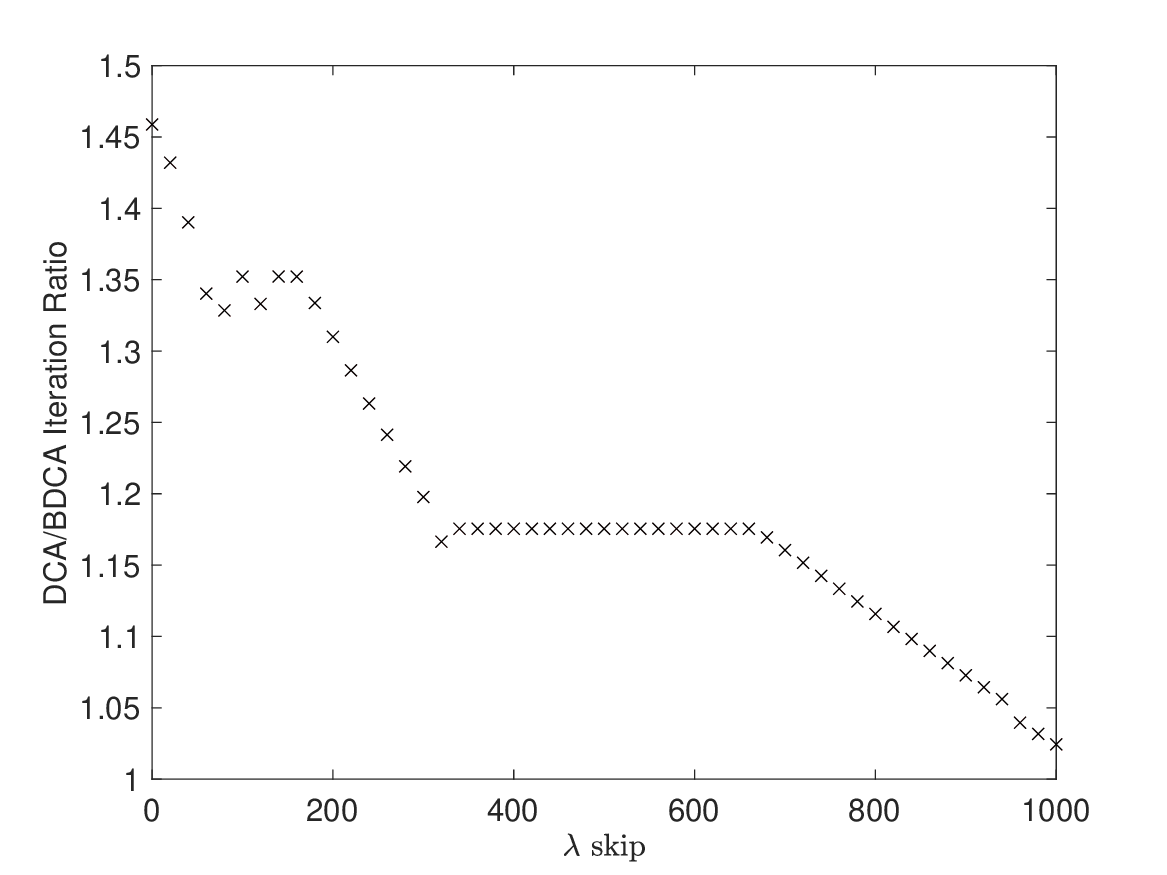}
	\end{subfigure}
	%\hfill
	\begin{subfigure}[b]{0.42\textwidth}
		\centering
		\includegraphics[width=\textwidth]{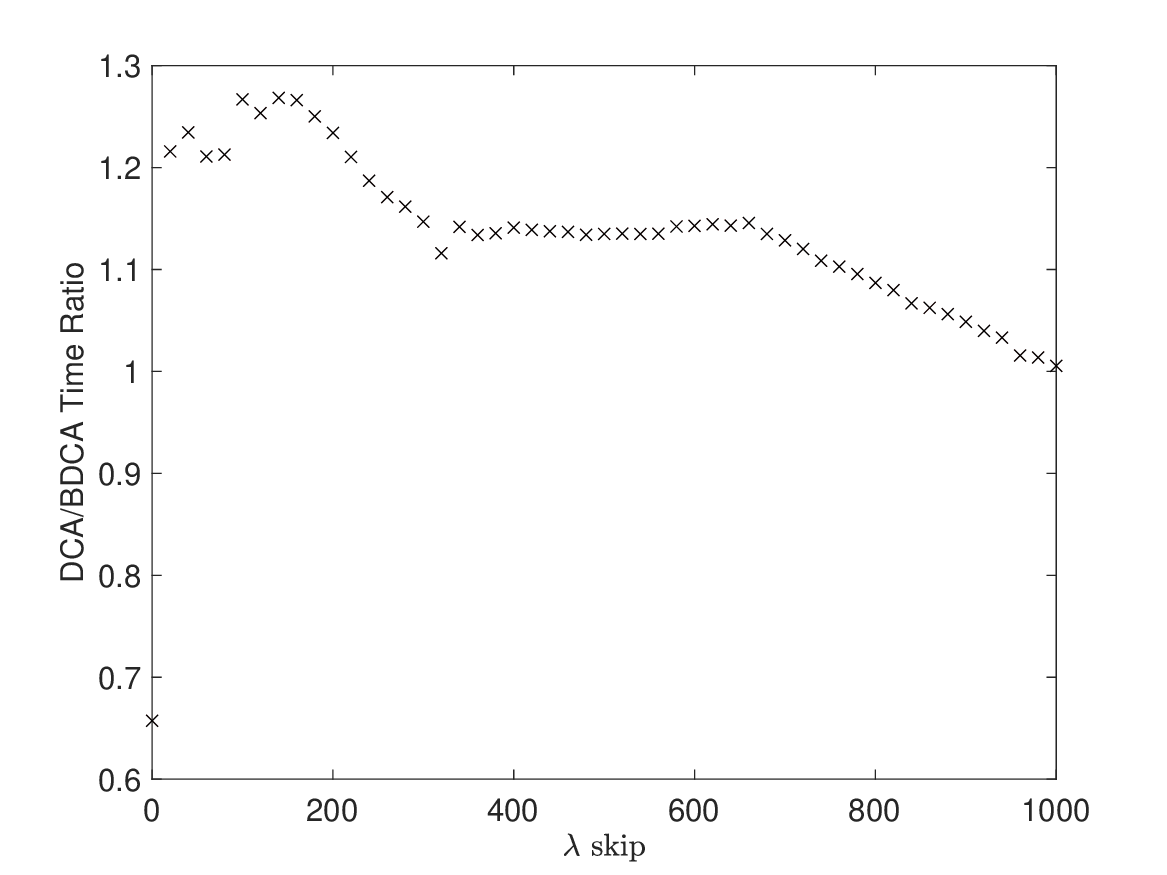}
	\end{subfigure}
	\caption{Average DCA/aBDCA iteration and time ratio for \Cref{example: traveling salesman} using varying values of $\lambda_{\text{skip}}$ in \cref{BDCA constrained - 2}.}
	\label{fig: 76 cities skipping plot}
\end{figure}

\newpage
\section{Conclusion}\label{Con}
In this paper, we provide details exploring the fundamental qualitative properties of the Generalized Multi-source Weber Problem using the Minkowski gauge function instead of the Euclidean norm. We show that the problem always admits a global solution and that the global solution set is compact. In addition, we introduce a concept of a local solution and provide necessary and sufficient conditions for its existence. From a numerical perspective, we propose an adaptive BDCA with skipping technique that improves both iteration counts and run-time. Unlike the results in \cite{Artacho2020}, adaptive BDCA struggles with GMWP. While reducing iterations compared to DCA, it suffers from slower run times due to the line search cost. Our new algorithm incorporates a skipping mechanism to avoid unnecessary line searches, which improves both iteration counts and run-time. The skipping parameter is robust, with a wide range of effective values for each problem. Future work will focus on developing an advanced skipping method for more complex problems and testing the algorithm on large-scale, higher-dimensional problems.

\end{document}